\documentclass[12pt,reqno]{smfart}
\usepackage{amsmath,amssymb,smfthm,stmaryrd,epic,enumerate}
\usepackage[frenchb]{babel}
\usepackage[latin1]{inputenc}
\usepackage{a4wide}

\usepackage{pstricks}

\usepackage{xypic}
\usepackage[active]{srcltx}
\usepackage{graphicx}

\usepackage{lmodern}

\usepackage{hyperref}

\NumberTheoremsIn{subsection}\SwapTheoremNumbers

\begin{document}
\newtheorem{prop-defi}[smfthm]{Proposition-DÈfinition}
\newtheorem{notas}[smfthm]{Notations}
\newtheorem{nota}[smfthm]{Notation}
\newtheorem{defis}[smfthm]{DÈfinitions}
\newtheorem{hypo}[smfthm]{HypothËse}

\def\Xm{{\mathbb X}}
\def\Um{{\mathbb U}}
\def\Am{{\mathbb A}}
\def\Fm{{\mathbb F}}
\def\Mm{{\mathbb M}}
\def\Nm{{\mathbb N}}
\def\Pm{{\mathbb P}}
\def\Qm{{\mathbb Q}}
\def\Zm{{\mathbb Z}}
\def\Dm{{\mathbb D}}
\def\Cm{{\mathbb C}}
\def\Rm{{\mathbb R}}
\def\Gm{{\mathbb G}}
\def\Lm{{\mathbb L}}
\def\Km{{\mathbb K}}
\def\Om{{\mathbb O}}
\def\Em{{\mathbb E}}

\def\BC{{\mathcal B}}
\def\QC{{\mathcal Q}}
\def\TC{{\mathcal T}}
\def\ZC{{\mathcal Z}}
\def\AC{{\mathcal A}}
\def\CC{{\mathcal C}}
\def\DC{{\mathcal D}}
\def\EC{{\mathcal E}}
\def\FC{{\mathcal F}}
\def\GC{{\mathcal G}}
\def\HC{{\mathcal H}}
\def\IC{{\mathcal I}}
\def\JC{{\mathcal J}}
\def\KC{{\mathcal K}}
\def\LC{{\mathcal L}}
\def\MC{{\mathcal M}}
\def\NC{{\mathcal N}}
\def\OC{{\mathcal O}}
\def\PC{{\mathcal P}}
\def\UC{{\mathcal U}}
\def\VC{{\mathcal V}}
\def\XC{{\mathcal X}}
\def\SC{{\mathcal S}}

\def\BF{{\mathfrak B}}
\def\AF{{\mathfrak A}}
\def\GF{{\mathfrak G}}
\def\EF{{\mathfrak E}}
\def\CF{{\mathfrak C}}
\def\DF{{\mathfrak D}}
\def\JF{{\mathfrak J}}
\def\LF{{\mathfrak L}}
\def\MF{{\mathfrak M}}
\def\NF{{\mathfrak N}}
\def\XF{{\mathfrak X}}
\def\UF{{\mathfrak U}}
\def\KF{{\mathfrak K}}
\def\FF{{\mathfrak F}}

\def \longmapright#1{\smash{\mathop{\longrightarrow}\limits^{#1}}}
\def \mapright#1{\smash{\mathop{\rightarrow}\limits^{#1}}}
\def \lexp#1#2{\kern \scriptspace \vphantom{#2}^{#1}\kern-\scriptspace#2}
\def \linf#1#2{\kern \scriptspace \vphantom{#2}_{#1}\kern-\scriptspace#2}
\def \linexp#1#2#3 {\kern \scriptspace{#3}_{#1}^{#2} \kern-\scriptspace #3}

\def \Ext{\mathop{\mathrm{Ext}}\nolimits}
\def \ad{\mathop{\mathrm{ad}}\nolimits}
\def \sh{\mathop{\mathrm{Sh}}\nolimits}
\def \irr{\mathop{\mathrm{Irr}}\nolimits}
\def \FH{\mathop{\mathrm{FH}}\nolimits}
\def \FPH{\mathop{\mathrm{FPH}}\nolimits}
\def \coh{\mathop{\mathrm{Coh}}\nolimits}
\def \res{\mathop{\mathrm{res}}\nolimits}
\def \op{\mathop{\mathrm{op}}\nolimits}
\def \rec {\mathop{\mathrm{rec}}\nolimits}
\def \art{\mathop{\mathrm{Art}}\nolimits}
\def \hyp {\mathop{\mathrm{Hyp}}\nolimits}
\def \cusp {\mathop{\mathrm{Cusp}}\nolimits}
\def \scusp {\mathop{\mathrm{Scusp}}\nolimits}
\def \Iw {\mathop{\mathrm{Iw}}\nolimits}
\def \JL {\mathop{\mathrm{JL}}\nolimits}
\def \speh {\mathop{\mathrm{Speh}}\nolimits}
\def \isom {\mathop{\mathrm{Isom}}\nolimits}
\def \Vect {\mathop{\mathrm{Vect}}\nolimits}
\def \groth {\mathop{\mathrm{Groth}}\nolimits}
\def \hom {\mathop{\mathrm{Hom}}\nolimits}
\def \deg {\mathop{\mathrm{deg}}\nolimits}
\def \val {\mathop{\mathrm{val}}\nolimits}
\def \det {\mathop{\mathrm{det}}\nolimits}
\def \rep {\mathop{\mathrm{Rep}}\nolimits}
\def \spec {\mathop{\mathrm{Spec}}\nolimits}
\def \fr {\mathop{\mathrm{Fr}}\nolimits}
\def \frob {\mathop{\mathrm{Frob}}\nolimits}
\def \ker {\mathop{\mathrm{Ker}}\nolimits}
\def \im {\mathop{\mathrm{Im}}\nolimits}
\def \Red {\mathop{\mathrm{Red}}\nolimits}
\def \red {\mathop{\mathrm{red}}\nolimits}
\def \aut {\mathop{\mathrm{Aut}}\nolimits}
\def \diag {\mathop{\mathrm{diag}}\nolimits}
\def \spf {\mathop{\mathrm{Spf}}\nolimits}
\def \Def {\mathop{\mathrm{Def}}\nolimits}
\def \twist {\mathop{\mathrm{Twist}}\nolimits}
\def \supp {\mathop{\mathrm{Supp}}\nolimits}
\def \Id {{\mathop{\mathrm{Id}}\nolimits}}
\def \lie {{\mathop{\mathrm{Lie}}\nolimits}}
\def \Ind{\mathop{\mathrm{Ind}}\nolimits}
\def \ind {\mathop{\mathrm{ind}}\nolimits}
\def \soc {\mathop{\mathrm{Soc}}\nolimits}
\def \top {\mathop{\mathrm{Top}}\nolimits}
\def \ker {\mathop{\mathrm{Ker}}\nolimits}
\def \coker {\mathop{\mathrm{Coker}}\nolimits}
\def \gal {{\mathop{\mathrm{Gal}}\nolimits}}
\def \Nr {{\mathop{\mathrm{Nr}}\nolimits}}
\def \rn {{\mathop{\mathrm{rn}}\nolimits}}
\def \tr {{\mathop{\mathrm{Tr~}}\nolimits}}
\def \Sp {{\mathop{\mathrm{Sp}}\nolimits}}
\def \st {{\mathop{\mathrm{St}}\nolimits}}
\def \sp{{\mathop{\mathrm{Sp}}\nolimits}}
\def \perv{\mathop{\mathrm{Perv}}\nolimits}
\def \tor {{\mathop{\mathrm{Tor}}\nolimits}}
\def \nrd {{\mathop{\mathrm{Nrd}}\nolimits}}
\def \nilp {{\mathop{\mathrm{Nilp}}\nolimits}}
\def \obj {{\mathop{\mathrm{Obj}}\nolimits}}
\def \cl {{\mathop{\mathrm{cl}}\nolimits}}

\def \rem{{\noindent\textit{Remarque:~}}}
\def \ext {{\mathop{\mathrm{Ext}}\nolimits}}
\def \End {{\mathop{\mathrm{End}}\nolimits}}

\def\semi{\mathrel{>\!\!\!\triangleleft}}
\let \DS=\displaystyle

\setcounter{secnumdepth}{3} \setcounter{tocdepth}{3}

\def \Fil{\mathop{\mathrm{Fil}}\nolimits}
\def \CoFil{\mathop{\mathrm{CoFil}}\nolimits}
\def \Fill{\mathop{\mathrm{Fill}}\nolimits}
\def \CoFill{\mathop{\mathrm{CoFill}}\nolimits}
\def\SF{{\mathfrak S}}
\def\PF{{\mathfrak P}}
\def \EFil{\mathop{\mathrm{EFil}}\nolimits}
\def \EFill{\mathop{\mathrm{EFill}}\nolimits}
\def \FP{\mathop{\mathrm{FP}}\nolimits}

\let \longto=\longrightarrow
\let \oo=\infty

\let \d=\delta
\let \k=\kappa

\newcommand{\marque}{\addtocounter{smfthm}{1}
{\smallskip \noindent \textit{\thesmfthm}~---~}}

\renewcommand\atop[2]{\ensuremath{\genfrac..{0pt}{1}{#1}{#2}}}

\title[Congruences automorphes et torsion]{Congruences automorphes et torsion dans la cohomologie d'un systËme local d'Harris-Taylor}

\alttitle{Automorphic congruences and torsion in the cohomology of Harris-Taylor local systems}

\author{Boyer Pascal}
\email{boyer@math.univ-paris13.fr}
\address{}

\frontmatter

\begin{abstract}
Cet article fait partie d'une sÈrie de papiers de l'auteur visant ‡ comprendre la 
cohomologie ‡ coefficients dans
$\bar \Zm_{l}$ de certaines variÈtÈs  de Shimura unitaires simples.
On utilise les rÈsultats de \cite{boyer-compositio} sur le calcul de la $\bar \Qm_l$-
cohomologie et ceux de \cite{boyer-ens} sur une version entiËre des rÈsultats de 
\cite{boyer-invent2} sur des filtrations du faisceau pervers des cycles proches,
pour produire d'une part des classes
de cohomologie de torsion dans les termes initiaux d'une suite spectrale calculant la 
cohomologie de la variÈtÈ de Shimura, et d'autre part
des congruences entre reprÈsentations automorphes.
\end{abstract}

\begin{altabstract}
This paper is a part of a sequence of works of the author to understand the $\bar \Zm_l$-cohomology of some simple unitary Shimura varieties. We use the results of
\cite{boyer-compositio} on the description of the $\bar \Qm_l$-cohomology groups, and
those of \cite{boyer-ens} about a $\bar \Zm_l$-version of the results of \cite{boyer-invent2} 
on filtrations of the perverse sheaf of nearby cycles, to produce on one part 
torsion classes in some of the initial
terms of a spectral sequence which calculate the cohomology groups of the
Shimura variety and, on a another part, some automorphic congruences.
\end{altabstract}



\altkeywords{Shimura varieties, formal modules, Langlands correspondences, Jacquet-Langlands
correspondences,
monodromy filtration, weight-monodromy conjecture, perverse sheaves, vanishing cycles}

\maketitle

\pagestyle{headings} \pagenumbering{arabic}


\section*{Introduction}
\renewcommand{\theequation}{\arabic{equation}}
\backmatter

Dans \cite{h-t}, les auteurs Ètudient la cohomologie d'une classe de variÈtÈs de Shimura 
unitaires $X$ de dimension $d-1$ sur son corps reflex $F$, attachÈes ‡ un groupe de 
similitudes $G/\Qm$. 
La description de la cohomologie d'une telle variÈtÈ 
de Shimura est donnÈe via la cohomologie de la fibre spÈciale $\overline X_v$
de $X$ en une place $v$ de $F$, ‡ coefficients dans le complexe des cycles Èvanescents. 
Dans \cite{h-t}, pour tout $1†\leq h \leq d$, la restriction des faisceaux des cycles
Èvanescents ‡ la strate de Newton de la fibre spÈciale $\overline X_v$ de dimension $d-h$,
est dÈcrite ‡ l'aide de certains 
systËmes locaux, dits d'Harris-Taylor dans \cite{boyer-invent2}, associÈs
‡ des reprÈsentations du groupe des inversibles $D_{v,h}^\times$ d'une algËbre ‡ division 
centrale sur le corps local $F_v$ et d'invariant $1/h$.

En ce qui concerne la torsion, contrairement ‡ \cite{h-t}, on ne peut pas se contenter 
d'arguments sur la somme alternÈe des groupes de cohomologie car
cette somme ne voit pas la torsion. Afin d'Ètudier chacun de ces groupes de cohomologie,
on commence par filtrer le faisceau pervers dÈfini par le complexe
des cycles Èvanescents, et on Ètudie la suite spectrale associÈe. Dans 
\cite{boyer-ens} on construit des filtrations \og entiËres \fg{} du faisceau pervers
des cycles Èvanescents qui gÈnÈralisent celles de \cite{boyer-invent2} de sorte qu'on est 
dans un premier temps amenÈ ‡ Ètudier la cohomologie des systËmes locaux 
d'Harris-Taylor. Pour ce qui concerne leurs quotients libres, ils sont  explicitÈs dans \cite{boyer-compositio} et 
le premier objectif de ce papier est la construction \S \ref{para-torsion} de classes de cohomologie de torsion. 
Une version simplifiÈe de la proposition \ref{prop-torsion} peut s'exprimer comme suit.

\smallskip

\textit{Soit $\pi_v$ une reprÈsentation irrÈductible cuspidale de $GL_g(F_v)$ dont la rÈduction modulo $l$,
a pour support supercuspidal un segment de Zelevinsky relatif ‡ une supercuspidale de $GL_{\bar g}(F_v)$
avec $\bar g$ un diviseur strict de $g$. Pour tout $t \geq 1$ tel que $tg+\bar g \leq d$, la torsion de la 
cohomologie ‡ support compact du systËme local d'Harris-Taylor associÈ, par la correspondance de
Jacquet-Langlands, ‡ la reprÈsentation de Steinberg gÈnÈralisÈe $\st_t(\pi_v)$, est non nulle.}

\smallskip

Ainsi nous disposons d'une
suite spectrale calculant la cohomologie entiËre d'une variÈtÈ de Shimura, dans laquelle
nous savons construire des classes de cohomologie de torsion dans ses termes initiaux
$E_1^{p,q}$; reste alors ‡ comprendre, soit comment cette torsion
disparait dans l'aboutissement, soit pourquoi elle y apparait, problÈmatique sur laquelle 
nous espÈrons revenir.

La construction de ces classes de torsion repose sur l'Ètude de reprÈsentations 
automorphes congruentes. GrossiËrement la problÈmatique est la suivante: Ètant donnÈs
\begin{itemize}
\item une $\bar \Qm_{l}$-reprÈsentation automorphe entiËre $\Pi$; 

\item un $\bar \Fm_l$-constituant irrÈductible $\tau_{v}$ de la rÈduction modulo $l$ de la 
composante locale $\Pi_{v}$ de $\Pi$ en une place $v$;

\item une $\bar \Qm_{l}$-reprÈsentation irrÈductible locale entiËre $\Pi'_{v}$, telle que
$\tau_v$ est un constituant de sa rÈduction modulo $l$,
\end{itemize}
existe-t-il une reprÈsentation automorphe $\Pi'$ de composante locale en $v$ isomorphe 
‡ $\Pi'_{v}$ et \og congruente \fg{} ‡ $\Pi$ en un sens ‡ prÈciser? Bien Èvidemment
nous ne traitons pas cette question dans cette gÈnÈralitÈ mais seulement quelques cas
particuliers dont le plus utile correspond,
cf. le corollaire \ref{coro-congru2}, au cas o˘ $\Pi'_{v}$ est non dÈgÈnÈrÈe et
$\Pi_{v}$ est de la forme $\speh_{s}(\pi_{v})$ avec $\pi_{v}$ non dÈgÈnÈrÈe.
Du cÙtÈ galoisien, on peut voir ce rÈsultat comme 
\og une augmentation de l'irrÈductibilitÈ \fg{} au sens o˘ partant d'une reprÈsentation 
galoisienne $\sigma$
on en construit une autre $\sigma'$ qui lui est \og congruente \fg{} et dont la restriction
$\sigma'_v$ au groupe de dÈcomposition en $v$ consiste ‡ changer un des 
facteurs de $\sigma_v$ d'une des deux faÁons dÈcrites ci-dessous:
\begin{itemize}
\item un facteur indÈcomposable de $\sigma_v$ est \og transformÈ \fg{} 
en un facteur irrÈductible de $\sigma'_v$;

\item un facteur semi-simple de $\sigma_v$ devient un facteur 
indÈcomposable de $\sigma'_v$.
\end{itemize} 
Une version simplifiÈe de la proposition \ref{prop-congru2} dans le cas trivial ‡ l'infini, 
illustrant le premier de ces cas, peut s'exprimer comme suit.

\smallskip

\noindent \textbf{Proposition} \rule[.1cm]{.6cm}{.2pt}

\textit{Soit $\pi'_v$ une reprÈsentation irrÈductible cuspidale de $GL_d(F_v)$ dont la rÈduction modulo $l$
a pour support supercuspidal un segment de Zelevinsky de longueur $s >1$ relativement ‡ une supercuspidale
$\bar \pi_v$ de $GL_{g}(F_v)$ avec $d=sg$ o˘ $\bar \pi_v$ est la rÈduction modulo $l$ d'une
reprÈsentation irrÈductible cuspidale $\pi_v$. \\
Alors Ètant donnÈe une reprÈsentation irrÈductible automorphe $\Pi$ de $G(\Am)$,
triviale ‡ l'infini, dont la composante locale en $v$
est $\Pi_v \simeq \st_s (\pi_v)$, il existe une reprÈsentation irrÈductible automorphe $\Pi'$ de $G(\Am)$, triviale
‡ l'infini, avec $\Pi'_v \simeq \pi'_v$, qui est congruente ‡ $\Pi_v$ au sens de la dÈfinition \ref{defi-faible},
i.e. il existe un ensemble fini $S$ de nombre premiers tel que pour tout premier $w \not \in S$, 
les reprÈsentations
$\Pi_w$ et $\Pi'_w$ sont non ramifiÈes et leurs caractËres de Satake sont congruents modulo $l$.
}

\smallskip

En ce qui concerne le passage d'un facteur semi-simple ‡ un facteur indÈcomposable, 
l'ÈnoncÈ donnÈ au corollaire \ref{coro-congru2} s'exprime comme ci-dessus avec
$\Pi_{v} \simeq \speh_s(\pi_v)$. Le lecteur notera que, contrairement au cas de $GL_2$ traitÈ en toute gÈnÈralitÈ 
dans \cite{billerey}, nous ne nous autorisons ni ‡ augmenter le niveau ni le poids.

Enfin au \S \ref{para-rep-congru-forte}, motivÈ par la construction de reprÈsentations
automorphes fortement congruentes, cf. \S \ref{para-ramifie} les cas ramifiÈs de la 
conjecture \ref{conj} et le lemme \ref{lem-congru-forte}, nous proposons la conjecture
\ref{conj-2} sur la torsion des systËmes locaux d'Harris-Taylor, dont la rÈduction modulo $l$ 
ne devrait dÈpendre que de la rÈduction modulo $l$ dudit systËme local.

\smallskip

L'auteur tient ‡ remercier chaleureusement 
le rapporteur anonyme pour ses nombreuses corrections et amÈliorations.

\tableofcontents

\mainmatter

\renewcommand{\theequation}{\arabic{section}.\arabic{subsection}.\arabic{smfthm}}

\section{Rappels sur les reprÈsentations}

Dans ce qui suit $l$ et $p$ dÈsignent deux nombres premiers distincts. Pour $q$ une puissance de $p$,
on note $e_l(q)$ l'ordre de $q$ dans $\Fm_l^\times$. Dans la suite $K$ dÈsigne une extension finie de $\Qm_p$, 
d'anneau des entiers $\OC_K$ et de corps rÈsiduel $\kappa$ de cardinal $q$. 
On dit que $l$ est \textit{banal} pour $GL_d(K)$ si $e_l(q)>d$. 
On fixe une racine carrÈe de $q$ dans $\bar \Zm_l$ et on note
$$\Xi:\frac{1}{2}\Zm \longto \bar \Qm_l^\times$$ 
le caractËre dÈfini par $\Xi(\frac{1}{2})=q^{1/2}$.

\subsection{ReprÈsentations entiËres de $GL_n(K)$ et leurs rÈduction modulo $l$}
\label{para-rap-rep}

On note $|-|$ la valeur absolue de $K$.

\begin{defi} \label{defi-upiv} \phantomsection
Deux reprÈsentations $\pi$ et $\pi'$ de $GL_n(K)$) sont dites inertiellement Èquivalentes et on note 
$\pi \sim^i \pi'$, s'il existe un caractËre $\xi:\Zm \longto \bar \Qm_l^\times$ tel que 
$$\pi \simeq \pi' \otimes \xi \circ \val \circ \det.$$
On note $e_{\pi}$ le cardinal de l'ensemble des caractËres
$\chi:\Zm \longto \bar \Qm_l^\times$, tels que  $\pi \otimes \chi \circ \val(\det) \simeq \pi$ et on pose
$$\pi \{ n \}:=\pi \otimes q^{-n \val \circ \det}.$$ 
\end{defi}

Pour une suite $\underline r=(0< r_1 < r_2 < \cdots < r_k=d)$, on note $P_{\underline r}$ le sous-groupe
parabolique de $GL_d$ standard associÈ au sous-groupe de Levi 
$$GL_{r_1}(K) \times GL_{r_2-r_1}(K) \times \cdots \times GL_{r_k-r_{k-1}}(K)$$ 
et $N_{\underline r}$ son radical unipotent. Le parabolique opposÈ ‡ $P_{\underline r}$ sera notÈ
$P_{\underline r}^{op}$.

\begin{nota}
Pour $\pi_1$ et $\pi_2$ des $\bar \Qm_l$-reprÈsentations de respectivement $GL_{n_1}(K)$ et
$GL_{n_2}(K)$, $\pi_1 \times \pi_2$ dÈsigne l'induite parabolique normalisÈe
$$\pi_1 \times \pi_2:=\ind_{P_{n_1,n_1+n_2}(K)}^{GL_{n_1+n_2}(K)} \pi_1 \{ n_2/2 \} \otimes \pi_2 \{-n_1/2\}.$$
Dans le cas o˘ $\pi_1$ et $\pi_2$ sont des reprÈsentations de $GL_{t_1g}(K)$ et
$GL_{t_2g}(K)$, on notera $\pi_1 \overrightarrow{\times} \pi_2$ l'induite parabolique $\pi_1\{ -t_2/2 \}
\times \pi_2 \{ t_1/2 \}$, l'entier $g$ Ètant sous-entendu.
\end{nota}

\rem les symboles $\times$ et $ \overrightarrow{\times}$ sont associatifs, le deuxiËme Ètant introduit
afin de minimiser l'Ècriture des torsions.

\begin{nota}
Pour $P=MN$ un parabolique de $GL_d$ de LÈvi $M$ et de radical unipotent $N$, 
l'espace des vecteurs $N(K)$-coinvariants d'une reprÈsentation admissible $\pi$ de $GL_d(K)$, 
est stable sous l'action de $M(K) \simeq P(K)/N(K)$. On notera $J_P(\pi)$ cette reprÈsentation tordue par 
$\delta_P^{-1/2}$: c'est un foncteur exact dit \emph{foncteur de Jacquet}.
\end{nota}

Rappelons qu'une reprÈsentation $\varrho$ de $GL_d(K)$ est dite \textit{cuspidale} 
si pour tout parabolique propre $P$, $J_{P}(\pi)$ est nul. Elle est dite 
\textit{supercuspidale} si elle n'est pas un sous-quotient 
d'une induite parabolique propre. Une reprÈsentation supercuspidale est cuspidale,
la rÈciproque Ètant valable pour les $\bar \Qm_l$-reprÈsentations.

\begin{notas} \label{defi-rep} (cf. \cite{zelevinski2} \S 9 et \cite{boyer-compositio} \S 1.4)
- Soient $g$ un diviseur de $d=sg$ et $\pi$
une reprÈsentation cuspidale irrÈductible de $GL_g(K)$. L'induite
$$\pi\{ \frac{1-s}{2} \} \times \pi \{ \frac{3-s}{2} \} \times \cdots \times \pi \{ \frac{s-1}{2} \}= 
\pi  \overrightarrow{\times} \cdots \overrightarrow{\times} \pi $$ 
possËde un unique quotient (resp. sous-espace) irrÈductible notÈ habituellement $\st_s(\pi)$ (resp.
$\speh_s(\pi)$); c'est une reprÈsentation de Steinberg (resp. de Speh) gÈnÈralisÈe.

- Plus gÈnÈralement,  l'induite
$$\st_t(\pi \{ \frac{1-s}{2} \} ) \times \cdots \times \st_t(\pi \{ \frac{s-1}{2} \} )$$
admet un unique sous-espace irrÈductible que l'on note $\speh_s \bigl ( \st_t(\pi) \bigr )$.

- Enfin, une reprÈsentation gÈnÈrique $\Pi$ de $GL_d(K)$ s'Ècrit comme une induite irrÈductible
$$\st_{t_1}(\pi_1) \times \cdots \times \st_{t_r}(\pi_r)$$
o˘ pour $i=1,\cdots, r$, les $\pi_i$ sont des reprÈsentations irrÈductibles cuspidales de $GL_{g_i}(K)$ et les
$t_i \geq 1$ tels que $\sum_{i=1}^r t_ig_i=d$. Alors la reprÈsentation
$$\speh_s \bigl ( \st_{t_1}(\pi_1) \bigr ) \times \cdots \times \speh_s \bigl (\st_{t_r}(\pi_r) \bigr )$$
est irrÈductible, on la note $\speh_s(\Pi)$.
\end{notas}

\rem dans \cite{boyer-alg}, la reprÈsentation $\st_s(\pi)$ (resp. $\speh_s(\pi)$) est notÈe
$[\overleftarrow{s-1}]_{\pi}$ (resp.  $[\overrightarrow{s-1}]_{\pi})$).
Toute reprÈsentation irrÈductible \emph{essentiellement de carrÈ intÈgrable} 
de $GL_d(K)$ est de la forme $\st_s(\pi)$ pour $\pi$ irrÈductible cuspidale de $GL_g(K)$ 
avec $d=sg$. 

Soient $\psi: \OC_K \longto R^\times$ un caractËre et $\underline \lambda=(\lambda_1 \geq \cdots \geq \lambda_r >0)$ une partition de $n$;
on dÈfinit alors un caractËre $\psi_{\underline \lambda}$ sur le radical unipotent $U$ du Borel standard de $GL_n(K)$, par la formule
$$\psi_{\underline \lambda}(u):=\psi(\sum_i u_{i,i+1})$$
o˘ la somme porte sur les indices $1 \leq i < n$ tels que $i \neq \lambda_1,\lambda_1+\lambda_2, \cdots$.

\begin{prop-defi} \cite{vigneras-livre} III.1.8 \\
Soit $\pi$ une $R$-reprÈsentation irrÈductible de $GL_n(K)$; il existe alors une unique partition $\underline{\lambda_\pi}$
maximale pour l'ordre de Bruhat, telle que le nombre fini
$$m(\pi,\underline \lambda):=\dim_{R} \hom_{RU}(\pi_{|U},\psi_{\underline \lambda})$$
est non nul. On dÈsigne cette partition comme \emph{le niveau de Whittaker} de $\pi$.
\end{prop-defi}

\begin{defi} 
Une reprÈsentation irrÈductible $\pi$ telle que $\underline{\lambda_\pi}=(n)$, est dite \emph{non dÈgÈnÈrÈe}; on a alors $m(\pi,(n))=1$.
\end{defi}

\rem toute $\bar \Qm_l$-reprÈsentation non dÈgÈnÈrÈe s'Ècrit comme une induite 
irrÈductible de la forme 
$$\st_{t_1}(\pi_1) \times \cdots \st_{t_r}(\pi_r)$$
o˘ les $\pi_i$ sont des reprÈsentations irrÈductibles cuspidales.

\begin{defi} Une $\bar \Qm_l$-reprÈsentation lisse de longueur finie $\pi$ de $GL_d(K)$ est dite \emph{entiËre} s'il existe une extension finie $E/ \Qm_l$ contenue dans
$\bar \Qm_l$, d'anneau des entiers 
$\OC_E$ et une $\OC_E$-reprÈsentation $L$ de $GL_d(K)$, qui est un 
$\OC_E$-module libre, telle que $\bar \Qm_l \otimes_{\OC_E} L \simeq \pi$
et tel que $L$ est un $\OC_E GL_n(K)$-module de type fini. 
Soit $\k_E$ le corps rÈsiduel de $\OC_E$, on dit que $\bar \Fm_l \otimes_{\kappa_E} 
\kappa_E \otimes_{\OC_E} L$ est la rÈduction modulo $l$ de $L$. 
\end{defi}

\rem le \textit{principe de Brauer-Nesbitt} affirme que 
la semi-simplifiÈe de $\bar \Fm_l \otimes_{\OC_E} L$ est une $\bar \Fm_l$-reprÈsentation de $GL_d(K)$ 
de longueur finie qui ne dÈpend pas du choix de $L$. Son image dans le groupe de
Grothendieck sera notÈe $r_l(\pi)$ et dite \textit{la rÈduction modulo $l$ de $\pi$}.

\noindent \textit{Exemples}: 
\begin{itemize}
\item d'aprËs  \cite{vigneras-induced} V.9.2 ou \cite{dat-jl} \S 2.2.3,
la rÈduction modulo $l$ de $\speh_s(\pi)$ est irrÈductible de sorte qu'‡ isomorphisme 
prËs, $\speh_s(\pi)$ possËde un unique rÈseau stable, 
cf. par exemple \cite{bellaiche-ribet} proposition 3.3.2 et la remarque qui suit.

\item En ce qui concerne $\st_s(\pi)$, la situation est en gÈnÈral plus complexe; 
signalons par exemple que dans \cite{boyer-alg}, nous en avons construit diffÈrents 
rÈseaux non isomorphes dits d'induction.
\end{itemize}


\begin{nota}  
Pour $\varrho$ une $\bar \Fm_l$-reprÈsentation cuspidale irrÈductible,
on note $\epsilon(\varrho)$ le cardinal de la droite de Zelevinsky de $\varrho$, i.e. de l'ensemble
$\{ \varrho\{ i\} ~/~ i \in \Zm \}$. On pose comme dans \cite{vigneras-induced} p.51, 
$$m(\varrho)= \left \{ \begin{array}{ll} \epsilon(\varrho) & \hbox{ si }\epsilon(\varrho)>1 \\
l & \hbox{sinon} \end{array} \right.$$
\end{nota}

\rem  $\epsilon(\varrho)$ un diviseur de $e_l(q)$.

\begin{prop} (\cite{vigneras-induced} III.5.14) \label{prop-stl}
Soit $\varrho$ une $\bar \Fm_l$-reprÈsentation irrÈductible cuspidale de $GL_g(K)$. L'induite parabolique
$$\varrho \overrightarrow{\times} \cdots \overrightarrow{\times} \varrho = \varrho\{\frac{1-s}{2} \} \times
\cdots \times \varrho \{ \frac{s-1}{2} \}$$ admet un unique sous-quotient non dÈgÈnÈrÈ que l'on note 
$\st_s(\varrho)$.
La reprÈsentation $\st_s(\varrho)$ est cuspidale si et seulement si
$$s=1,m(\varrho),m(\varrho)l,\cdots,m(\varrho) l^u,\cdots$$
La rÈunion de ces derniËres avec les supercuspidales forment l'ensemble des reprÈsentations
cuspidales.
\end{prop}

\rem en gÈnÈral, pour $\pi$ une $\bar \Qm_l$-reprÈsentation cuspidale telle que 
$r_l(\pi)=\varrho$, la rÈduction modulo $l$ de $\st_s(\pi)$ est strictement plus grande 
$\st_s(\varrho)$, cf. \cite{boyer-alg} \S 3; 
en revanche, comme cas particulier de la proposition suivante, la multiplicitÈ de 
$\st_s(\varrho)$ y est Ègale ‡ $1$.

\begin{prop} (cf. \cite{vigneras-induced} V.9.2) \label{prop-red-Whittaker} \\
Soit $\pi$ une $\bar \Qm_l$-reprÈsentation irrÈductible entiËre de $GL_d(K)$
de niveau de Whittaker $\underline \lambda$. Parmi les constituants irrÈductible de
la rÈduction modulo $l$ de $\pi$, il y en a exactement un de niveau de Whittaker 
$\underline \lambda$, tous les autres Ètant de niveau strictement infÈrieur.
\end{prop}

\rem ainsi tout sous-quotient irrÈductible de la rÈduction modulo $l$ de $\speh_s(\pi)$
avec $\pi$ non dÈgÈnÈrÈe est de longueur de Whittaker supÈrieure ou Ègale ‡ $s$.

\begin{nota} \label{nota-varrho}\phantomsection
Pour $\varrho$ une reprÈsentation irrÈductible supercuspidales, on notera pour tout $u \geq 0$,
$\varrho_u$ la reprÈsentation cuspidale $\st_{m(\varrho)l^u}(\varrho)$ de la proposition prÈcÈdente
et on pose $\varrho_{-1}:=\varrho$.
\end{nota}

\begin{prop} (cf. \cite{vigneras-livre} III.5.10) \\ 
Soit $\varrho$ une reprÈsentation irrÈductible cuspidale de $GL_g(K)$; il 
existe alors un relËvement $\pi$ irrÈductible cuspidal de $\varrho$, i.e. tel que
$r_l(\pi)=\varrho$. 
\end{prop}

Donnons quelques prÈcisions sur l'ensemble des relËvements d'une cuspidale
donnÈe. Rappelons, cf. \cite{b-he} \S 5.5, que toute $\bar \Qm_l$-reprÈsentation irrÈductible
cuspidale est bijectivement \og attachÈe \fg{} 
‡ la donnÈe d'une extension sauvagement ramifiÈe $E \subset P$ et ‡
une orbite rÈguliËre de $\Delta \backslash X_1(E)$ o˘:
\begin{itemize}
\item $E$ est une extension modÈrÈment ramifiÈe de $K$ de degrÈ $\leq d$;

\item $X_1(E)$ est l'ensemble des caractËres modÈrÈment ramifiÈs de $E^\times$;

\item $\Delta$ est le groupe de Galois de $E/K$.
\end{itemize}
Si $\chi_E$ est un reprÈsentant d'une telle orbite associÈ ‡ un relËvement cuspidal $\pi$ de
$\varrho$ alors les autres relËvements cuspidaux, ‡ torsion par un caractËre non ramifiÈ prËs,
seront associÈs ‡ $\chi_E \otimes \zeta_l$
o˘ $\zeta_l$ est un caractËre de $E^\times$ d'ordre une puissance de $l$.

\rem en particulier si $l$ est \textit{banal} pour $GL_d(K)$, toute
$\bar \Fm_l$-reprÈsentation irrÈductible cuspidale admet, ‡ torsion prËs, 
un unique relËvement sur $\bar \Qm_l$.

\subsection{ReprÈsentations de $D_{K,d}^\times$}

\begin{defi} Soit $D_{K,d}$ l'algËbre ‡ division centrale sur $K$ d'invariant $1/d$
et d'ordre maximal $\DC_{K,d}$. On identifiera
$$D_{K,h}^\times / \DC_{K,h}^\times \longto \Zm$$ 
au moyen de l'opposÈ de la valuation de la norme rÈduite $\rn$.
\end{defi}

On rappelle que la correspondance de Jacquet-Langlands locale est une bijection $\JL$ entre les reprÈsentations
irrÈductibles admissibles de $D_{K,h}^\times$ et les reprÈsentations irrÈductibles admissibles essentiellement de carrÈ
intÈgrable de $GL_{h}(K)$.

\begin{defi} Pour $\pi$ une $\bar \Qm_l$-reprÈsentation irrÈductible cuspidale de $GL_g(K)$ 
et $t \geq 1$,
$\pi[t]_D$ dÈsignera la reprÈsentation $\JL^{-1}(\st_t(\pi))^\vee$ de $D_{K,tg}^\times$.
\end{defi}

Soit $\tau=\pi[s]_{D}$ une $\bar \Qm_l$-reprÈsentation irrÈductible $l$-entiËre de $D_{K,d}^\times$;
avec les notations de \ref{nota-varrho}, la rÈduction modulo $l$ de $\pi$ est de
la forme $\st_{m(\tau)}(\varrho)$ pour $\varrho$
une $\bar \Fm_l$-reprÈsentation irrÈductible supercuspidale de $GL_e(K)$.
On note alors $\iota$ l'image de $\speh_{s}(\varrho^\vee)$ par la correspondance de Jacquet-Langlands modulaire
dÈfinie par J.-F. Dat au \S 1.2.4 de \cite{dat-jl}. Autrement dit si $\pi_\varrho$ est un relËvement cuspidal
de $\varrho$, i.e. une $\bar \Qm_l$-reprÈsentation entiËre irrÈductible cuspidale de $GL_e(K)$ dont la rÈduction
modulo $l$ est Ègale ‡ $\varrho$, alors $\iota=r_l \bigl ( \pi_\varrho[s]_D \bigr )$.

\begin{prop} \phantomsection \label{prop-red-D} (cf. \cite{dat-jl} proposition 2.3.3)
La rÈduction modulo $l$ de $\tau$ est avec les notations prÈcÈdentes de la forme
$$\iota \{-\frac{m(\tau)-1}{2} \} \oplus \iota \{-\frac{m(\tau)-3}{2} \} \oplus \cdots \oplus \iota \{ \frac{m(\tau)-1}{2} \}$$
o˘ $\iota \{ n \}$ dÈsigne $\iota \otimes q^{-n \val \circ \nrd}$.
\end{prop}

\subsection{ReprÈsentations automorphes cohomologiques de $G(\Am)$}
\label{para-defG}

Soient $E/\Qm$ une extension quadratique imaginaire pure, $F^+/\Qm$ une extension
totalement rÈelle dont on fixe un plongement rÈel $\tau:F^+ \hookrightarrow \Rm$;
on pose $F=F^+ E$ le corps CM associÈ. Soit $B$ une algËbre ‡ 
division centrale sur $F$ de dimension $d^2$ telle qu'en toute place $x$ de $F$,
$B_x$ est soit dÈcomposÈe soit une algËbre ‡ division et on suppose $B$ 
munie d'une involution de
seconde espËce $*$ telle que $*_{|F}$ est la conjugaison complexe $c$. Pour
$\beta \in B^{*=-1}$, on note $\sharp_\beta$ l'involution $x \mapsto x^{\sharp_\beta}=\beta x^*
\beta^{-1}$ et $G/\Qm$ le groupe de similitudes, notÈ $G_\tau$ dans \cite{h-t}, dÈfini
pour toute $\Qm$-algËbre $R$ par 
$$
G(R)  \simeq   \{ (\lambda,g) \in R^\times \times (B^{op} \otimes_\Qm R)^\times  \hbox{ tel que } 
gg^{\sharp_\beta}=\lambda \}
$$
avec $B^{op}=B \otimes_{F,c} F$. 
Si $x$ est une place de $\Qm$ dÈcomposÈe $x=yy^c$ dans $E$ alors 
$$G(\Qm_x) \simeq (B_y^{op})^\times \times \Qm_x^\times \simeq \Qm_x^\times \times
\prod_{z_i} (B_{z_i}^{op})^\times,$$
o˘ $x=\prod_i z_i$ dans $F^+$.

\begin{nota} \label{nota-GFv}
Pour $x$ une place de $\Qm$ dÈcomposÈe dans $E$ et $z$ une place de $F^+$ au dessus
de $x$, on notera $G(F_z)$ le facteur $(B_z^{op})^\times$ de $G(\Qm_x)$ et
$G(\Am^z)$ pour $G(\Am)$ auquel on Ùte le facteur $(B_z^{op})^\times$.
De mÍme pour $T$ un ensemble de places de $\Qm$ et $x \in T$, on notera 
$T-\{ z \}$ pour dÈsigner la rÈunion des places de $T$ distinctes de $x$ avec les
places de $F$, autres que $z$, au dessus de $x$.
\end{nota}

\rem pour $x=yy^c$ dÈcomposÈ dans $E$, les places de $F^+$ au dessus de $x$ s'identifie avec les places
de $F$ au dessus de $y$.

Dans \cite{h-t}, les auteurs justifient l'existence d'un $G$ comme ci-dessus tel qu'en outre:
\begin{itemize}
\item si $x$ est une place de $\Qm$ qui n'est pas dÈcomposÈe dans $E$ alors
$G(\Qm_x)$ est quasi-dÈployÈ;

\item les invariants de $G(\Rm)$ sont $(1,d-1)$ pour le plongement $\tau$ et $(0,d)$ pour les
autres. 
\end{itemize}
On suppose aussi que $p$ est dÈcomposÈe
$p=uu^c$ dans $E$ et on note, cf. la remarque ci-dessus, $v=v_1,v_2,\cdots,v_r$, les places de $F$ au dessus de 
$u$ avec $(B_v^{op})^\times \simeq GL_d(F_v)$.
Pour de coupables raisons de commoditÈ, on fixe un isomorphisme $\iota_{l}:\bar \Qm_{l} \simeq \Cm$.

\begin{defi} \label{defi-automorphe}\phantomsection
Soit $\xi$ une $\Cm$-reprÈsentation irrÈductible algÈbrique de dimension finie de $G$. 
Une $\Cm$-reprÈsentation irrÈductible $\Pi_{\oo}$ de $G(\Am_{\oo})$ est dite $\xi$-cohomologique s'il
existe un entier $i$ tel que
$$H^i((\lie G(\Rm)) \otimes_\Rm \Cm,U_\tau,\Pi_\oo \otimes \xi^\vee) \neq (0)$$
o˘ $U_\tau$ est un sous-groupe compact modulo le centre de $G(\Rm)$, maximal, cf. \cite{h-t} p.92. 
On notera $d_\xi^i(\Pi_\oo)$ la dimension de ce groupe de cohomologie.
\end{defi}

\begin{defi}
Soit $\xi$ une $\bar \Qm_l$-reprÈsentation irrÈductible algÈbrique de dimension finie de $G$. 
Une $\bar \Qm_{l}$-reprÈsentation irrÈductible $\Pi^{\oo}$ de $G(\Am^{\oo})$ sera dit automorphe 
$\xi$-cohomologique s'il existe une $\Cm$-reprÈsentation cohomologique $\Pi_\oo$ de $G(\Am_\oo)$ telle que
$\iota_{l}\Bigl ( \Pi^{\oo} \Bigr ) \otimes \Pi_{\oo}$ est une $\Cm$-reprÈsentation automorphe de $G(\Am)$.
\end{defi}

\begin{defi}
(cf. \cite{vigneras-langlands} appendice A) 
\\
Une $\bar \Qm_l$-reprÈsentation $\pi^\oo$ de $G(\Am^\oo)$ est dite \emph{entiËre}
s'il existe un $\bar{\Zm}_l$-module projectif $\Lambda$ muni d'une action de $G(\Am^\oo)$
tel que $\Lambda \otimes_{\bar \Zm_l} \bar \Qm_l \simeq \pi^\oo$ et $\Lambda^K$
est de type fini pour tout sous-groupe compact ouvert $K$ de $G(\Am^\oo)$.
\end{defi}

\rem d'aprËs \cite{vigneras-langlands} \S 3.3, pour toute reprÈsentation automorphe $\pi$ de
$G(\Am)$ triviale ‡ l'infini, sa partie finie $\pi^\oo$ est entiËre.

\section{Rappels sur les systËmes locaux d'Harris-Taylor}

\subsection{VariÈtÈs de Shimura unitaires simples}

Pour tout sous-groupe compact assez petit $U^p$ au sens du bas de la page 90 de \cite{h-t}, 
de $G(\Am^{\oo,p})$ et $m=(m_1,\cdots,m_r) \in \Zm_{\geq 0}^r$,  on pose
\addtocounter{smfthm}{1}
\begin{equation} \label{eq-zp}
U^p(m)=U^p \times \Zm_p^\times \times \prod_{i=1}^r \ker ( \OC_{B_{v_i}}^\times \longto
(\OC_{B_{v_i}}/v_i^{m_i})^\times )
\end{equation}
et on note $\IC$ l'ensemble des sous-groupes compacts assez petits de la forme $U^p(m)$. ¿ la donnÈe de
$G$ et $\IC$, on associe un \textbf{schÈma de Hecke} $X_\IC=(X_I)_{I \in \IC}$ sur $\spec \OC_v$ au sens de 
\cite{boyer-invent2} et on note $\overline{X_\IC}$ sa fibre spÈciale gÈomÈtrique sur $\bar \Fm_p$.

Avec les notations de \cite{boyer-invent2}, pour tout $1 \leq h \leq d$, on dispose de sous-$\bar \Fm_p$-schÈmas 
de Hecke pour $\Gm=G(\Am^\oo)$, fermÈs (resp. localement fermÈs)
notÈs $\overline X_\IC^{\geq h}$ (resp. $\overline X_\IC^{=h}$) de pure dimension $d-h$ et 
on note
$$i^{h}_\IC:\overline X_\IC^{\geq h} \hookrightarrow \overline X_\IC=
\overline X_\IC^{\geq 1}, \qquad
j^{\geq h}_\IC:\overline X_{\IC}^{=h} \hookrightarrow \overline X_\IC^{\geq h}.$$

\rem afin de ne pas multiplier les notations, on notera aussi
$i^h$ pour l'inclusion fermÈe $\overline X_\IC^{\geq h} \hookrightarrow 
\overline X_\IC^{\geq h+1}$.

Pour tout $1 \leq h< d$, les strates
$\overline X_\IC^{=h}$ sont gÈomÈtriquement induites sous l'action du parabolique $P_{h,d}(F_v)$ au sens o˘ il
existe un sous-schÈma fermÈ $\overline X_{\IC,1}^{=h}$ de Hecke pour $\Gm=G(\Am^{\oo,v}) \times P_{h,d}(F_v)$ tel que:
$$\overline X_\IC^{=h} \simeq \overline X_{\IC,1}^{=h} \times^{P_{h,d}(F_v)} GL_d(F_v)$$
On note $\overline X_{\IC,1}^{\geq h}$ l'adhÈrence de $\overline X_{\IC,1}^{=h}$ dans 
$\overline X_{\IC}^{\geq h}$. On rappelle que
$G(\Am^{\oo,v}) \times P_{h,d}(F_v)$ agit ‡ travers son quotient $G(\Am^{\oo,v}) \times \Zm \times GL_{d-h}(F_v)$ donnÈ par l'application $\left ( \begin{array}{cc} g_v^c & * \\ 0 & g_v^{et} \end{array} \right ) \mapsto (v(\det g_v^c),g_v^{et})$.
Par ailleurs l'action d'un ÈlÈment $w_v \in W_v$ est donnÈe par l'action de $-\deg(w_v)$.

\rem les points gÈomÈtriques du $\bar \Fm_p$-schÈma de Hecke $\overline X_{\IC}^{=d}$ de 
dimension nulle sont dits supersinguliers.

\begin{defi} Soit  $H_0/\Qm$ le groupe algÈbrique forme intÈrieure de $G$ telle que
$H_0(\Rm)$ est compact et $H_0(\Am^\oo) \simeq G(\Am^{\oo,p}) \times D_{v,d}^\times \times \prod_{i=2}^r (B_{v_i}^{op})^\times$.
\end{defi}

\subsection{SystËmes locaux d'Harris-Taylor}

Afin de dÈcrire les restrictions aux strates de Newton $\overline X_\IC^{=h}$ 
du complexe des cycles Èvanescents sur $\overline X_\IC$, les auteurs de \cite{h-t} ont construit des systËmes locaux
sur ces strates dont nous rappelons les notations suivant \cite{boyer-invent2}. ¿ toute reprÈsentation irrÈductible
admissible $\tau_v$ de $D_{v,h}^\times$, Harris et Taylor associent 
un systËme local de Hecke $\FC_{\tau_v,\IC,1}$ sur $\bar X_{\IC,1}^{=h}$ 
avec une action de
$G(\Am^{\oo,p}) \times \Qm_p^\times \times P_{h,d}(F_v) \times 
\prod_{i=2}^r (B_{v_i}^{op})^\times  \times \Zm$
qui d'aprËs \cite{h-t} p.136, se factorise par $G^{(h)}(\Am^\oo)/\DC_{F_v,h}^\times$ via
\addtocounter{smfthm}{1}
\begin{equation} \label{eq-action-tordue}
(g^{\oo,p},g_{p,0},c,g_v,g_{v_i},k) \mapsto (g^{p,\oo},g_{p,0}q^{k-v (\det g_v^c)}, 
g_v^{et},g_{v_i}, \delta).
\end{equation}
o˘ $G^{(h)}(\Am^\oo):=G(\Am^{\oo,p}) \times \Qm_p^\times \times GL_{d-h}(F_v) \times
\prod_{i=2}^r (B_{v_i}^{op})^\times \times D_{F_v,h}^\times$,
$g_v=\left ( \begin{array}{cc} g_v^c & * \\ 0 & g_v^{et} \end{array} \right )$ et
$\delta \in D_{v,h}^\times$ est tel que $v(\rn (\delta))=k+v(\det g_v^c)$.
On note $\FC_{\tau_v,\IC}$ le faisceau sur $\bar X_{\IC}^{=h}$ induit associÈ:
$$\FC_{\tau_v,\IC}:= \FC_{\tau_v,\IC,1} \times_{P_{h,d}(F_v)} GL_d(F_v).$$

Pour $\pi_v$ une reprÈsentation irrÈductible cuspidale de $GL_g(F_v)$ et $t$ un entier strictement positif tel que $tg \leq d$,
on introduit suivant \cite{boyer-invent2} la notation $\FC(\pi_v,t)_1$ (resp. $\FC(\pi_v,t)$) pour dÈsigner 
le faisceau de Hecke sur $\bar X_{\IC,1}^{=tg}$ (resp. $\bar X_\IC^{=tg}$)
prÈcÈdemment notÈ $\FC_{\pi_v[t]_D,\IC,1}$ (resp. $\FC_{\pi_v[t]_D,\IC}$).
On utilisera ponctuellement la notation $HT(\pi_v,\Pi_t)$ pour dÈsigner le $W_v$-faisceau pervers de Hecke sur
$\bar X_\IC^{=tg}$ pour $G(\Am^{\oo})$ dÈfini par
$$HT(\pi_v,\Pi_t):=\Bigl ( \FC(\pi_v,t)_1[d-tg] \otimes \Xi^{\frac{tg-d}{2}} \otimes \Pi_t \Bigr ) \times_{P_{h,d}(F_v)} GL_d(F_v),$$
o˘ $\Pi_t$ une reprÈsentation de $GL_{tg}(F_v)$ et o˘ on renvoie ‡ \cite{boyer-invent2} pour une description explicite
des opÈrateurs de Hecke.


\begin{nota} \label{nota-HTGamma}
On Ècrira $HT_{\bar \Zm_l}$ pour dÈsigner un $\bar \Zm_l$-rÈseau
du systËme local d'Harris-Taylor associÈ, que l'on ne souhaite pas prÈciser. 
Dans la suite, on supposera un rÈseau stable de $\pi_v[t]_D$ fixÈ
et pour $\Gamma$ un rÈseau stable de $\Pi_t$, on notera, $HT_\Gamma(\pi_v,\Pi_t)$
le $\bar \Zm_l$-rÈseau du systËme local d'Harris-Taylor associÈ au rÈseau
produit tensoriel stable de $\pi_v[t]_D \otimes \Pi_t$.
\end{nota}

\subsection{Filtrations de stratification entiËres}

Rappelons, cf. \cite{juteau} \S 1.3.1, qu'une thÈorie de torsion sur une catÈgorie abÈlienne
$\AC$ est un couple $(\TC,\FC)$ de sous-catÈgories pleines tel que:
\begin{itemize}
\item pour tout objet $T$ dans $\TC$ et $F$ dans $\FC$, on a
$$\hom_\AC(T,F)=0;$$

\item pour tout objet $A$ de $\AC$, il existe des objets $A_\TC$ et $A_\FC$ de respectivement $\TC$ et $\FC$,
ainsi qu'une suite exacte courte
$$0 \rightarrow A_\TC \longrightarrow A \longrightarrow A_\FC \rightarrow 0.$$
\end{itemize}
Pour $\Om$ un anneau de valuation discrËte et $\varpi$ une uniformisante, dans
une catÈgorie abÈlienne $\AC$ qui est $\Om$-linÈaire, un objet $A$ de $\AC$ est dit \textit{de torsion}
(resp. \textit{libre}, resp. \textit{divisible}) si $\varpi^N 1_A$ est nul pour un certain entier $N$ (resp. $\varpi.1_A$
est un monomorphisme, resp. un Èpimorphisme). On note alors $\TC$ (resp. $\FC$, resp. $\QC$)
la sous-catÈgorie pleine des objets de torsion (resp. libres, resp. divisibles) de $\AC$. Si $\AC$
est noethÈrienne (resp. artinienne) alors $(\TC,\FC)$ (resp. $(\QC,\TC)$) est une thÈorie de torsion sur $\AC$.

Pour $X$ un $\Fm_q$-schÈma de type fini, la catÈgorie dÈrivÈe $D_c^b(X,\Om)$ est munie
de sa $t$-structure perverse $p$ dÈfinie par:
$$\begin{array}{l}
A \in \lexp p D^{\leq 0}(X,\Om)
\Leftrightarrow \forall x \in X,~h^k i_x^* A=0,~\forall k >- \dim \overline{\{ x \} } \\
A \in \lexp p D^{\geq 0}(X,\Om) \Leftrightarrow \forall x \in X,~h^k i_x^! A=0,~\forall k <- \dim \overline{\{ x \} }
\end{array}$$
o˘ $i_x:\spec \kappa(x) \hookrightarrow X$. On note $\lexp p \FP(X,\Om)$ le c{\oe}ur de cette $t$-structure
et $\lexp p h$ les foncteurs cohomologiques: 
c'est une catÈgorie abÈlienne noethÈrienne et $\Om$-linÈaire munie de la thÈorie de torsion $(\TC,\FC)$
avec les notations prÈcÈdentes. On munit alors $D_c^b(X,\Om)$ d'une nouvelle $t$-structure notÈe $p+$ et
dÈfinie par, cf. \cite{boyer-ens} corollaire 1.1.4 et \S 1.4.2:
$$\begin{array}{l}
A \in \lexp {p+} D^{\leq 0}(X,\Om) \Leftrightarrow \forall x \in X,
\left \{ \begin{array}{ll} h^i i_x^* A=0, & \forall i >- \dim \overline{\{ x \} } +1 \\
h^{-\dim \overline{\{ x \} } +1} i_x^* A & \hbox{de torsion} \end{array} \right. \\
A \in \lexp {p+} D^{\geq 0}(X,\Om) \Leftrightarrow \forall x \in X,
\left \{ \begin{array}{ll} h^i i_x^! A=0, & \forall i <- \dim \overline{\{ x \} } \\
h^{-\dim \overline{\{ x \} }} i_x^! A & \hbox{libre} \end{array} \right.
\end{array}$$
dont on notera $\lexp {p+} \FP(X,\Om)$ le c{\oe}ur et $\lexp {p+} h^i$ les foncteurs cohomologiques.
C'est une catÈgorie abÈlienne artinienne et $\Om$-linÈaire. Ces constructions passent ‡ la limite sur
l'anneau $\bar \Zm_l$. 

\rem les objets libres de $\lexp p \FP(X,\Om)$, qui est une catÈgorie abÈlienne $\Om$-linÈaire, sont exactement ceux de $\lexp p \FP(X,\Om) \cap
\lexp {p+} \FP(X,\Om)$: ils sont appelÈs des faisceaux pervers sans torsion.

\medskip

Pour une immersion ouverte
$j:U \hookrightarrow X$, on dispose alors de deux notions de prolongements intermÈdiaires
$\lexp p j_{!*}$ et $\lexp {p+} j_{!*}$. 

\begin{lemm} (cf. \cite{boyer-ens} proposition 1.4.4) \label{lem-h1nul0}
Si $j:U \hookrightarrow X$ est affine alors pour tout faisceau pervers $P$ sans torsion, 
$$j_* P=\lexp p j_* P=\lexp {p+} j_* P \hbox{ et } j_! P=\lexp p j_! P=\lexp {p+} j_! P$$
sont des faisceaux pervers sans torsion.
\end{lemm}

\rem du triangle distinguÈ (cf. \cite{juteau} 2.42)
$$\lexp p j_! P \longrightarrow \lexp {p+} j_! P \longrightarrow \lexp p i_*
\lexp p h^{-1}_{tors} i^* j_* P [1] \leadsto$$
o˘ pour $A$ un faisceau pervers, $A_{tors}$ dÈsigne son sous-faisceau pervers
de torsion, on en dÈduit que, pour $j$ affine, $\lexp p i_* \lexp p h^{-1}_{tors} i^* j_* P$ est nul.

Dans \cite{boyer-ens}, en utilisant la stratification de Newton, nous donnons l'analogue entier
de 4.3.1 et 5.4.1 de \cite{boyer-invent2} qui s'exprime comme suit.


On note $\Fm(-)$ le foncteur de rÈduction modulaire $\Fm \otimes_\Om^\Lm (-)$; ce dernier ne commute pas
aux foncteurs de troncations. Ainsi d'aprËs \cite{juteau} 2.51, pour $P$
un faisceau pervers sans torsion, on a les triangles distinguÈs suivants
$$\begin{array}{l}
\Fm  \lexp p j^{\geq tg}_! P \rightarrow \lexp p j^{\geq tg}_! \Fm P \rightarrow 
\lexp p h^{-1} \Fm \lexp p i^{tg+1}_* ~ \lexp p
h^{-1}_{tors}i^{tg+1,*}j^{\geq tg}_*P  [2] \leadsto \\
\lexp p j^{\geq tg}_! \Fm P \rightarrow \Fm \lexp {p+} j^{\geq tg}_! P \rightarrow 
\lexp p h^0 \Fm \lexp p i^{tg+1}_* ~ \lexp p
h^{-1}_{tors} i^{tg+1,*} j^{\geq tg}_* P [1] \leadsto 
\end{array}$$
de sorte que $j^{\geq tg}$ Ètant affine, $\lexp p h^{-1}_{tors} i^{tg+1,*} j^{\geq tg}_*P $ est nul et donc
\addtocounter{smfthm}{1}
\begin{equation} \label{eq-h1nul} 
\Fm  \lexp p j^{\geq tg}_! P \simeq \lexp p j^{\geq tg}_! \Fm P \simeq \Fm \lexp {p+} j^{\geq tg}_! P.
\end{equation}

En ce qui concerne les extensions intermÈdiaires, d'aprËs loc. cit. on a
\addtocounter{smfthm}{1}
\begin{equation} \label{eq-Fjp}
\Fm \lexp p j^{\geq tg}_{!*} P \rightarrow \lexp p j^{\geq tg}_{!*} \Fm P \rightarrow 
\lexp p h^{-1} \Fm \lexp p i^{tg+1}_* ~
\lexp p h^0_{tors} i^{tg+1,*} j^{\geq tg}_* P [1] \leadsto 
\end{equation}
\addtocounter{smfthm}{1}
\begin{equation} \label{eq-Fjp+}
\lexp p j^{\geq tg}_{!*} \Fm P \rightarrow \Fm \lexp {p+} j^{\geq tg}_{!*} P \rightarrow 
\lexp p h^0 \Fm \lexp p i^{tg+1}_* ~ \lexp p h^0_{tors} i^{tg+1,*}j^{\geq tg}_* P \leadsto 
\end{equation}

\section{ComplÈments sur la cohomologie}

\subsection{Notations dans les groupes de Grothendieck}

Dans la suite on fixe une reprÈsentation complexe irrÈductible $\xi$ de dimension finie de $G$
 ce qui d'aprËs \cite{h-t} p.96, fournit un systËme local $\LC_\xi$.

\begin{defi} 
…tant donnÈ un faisceau $\FC$ sur $\bar X_\IC$, on notera $\FC_\xi$ le faisceau $\FC \otimes \LC_\xi$.
\end{defi}

\rem d'aprËs \cite{h-t} p.98 et p.150, $\LC_\xi$ est un $W_v$-faisceau de Hecke 
pour $G(\Am^\oo)$ o˘ $g^\oo$ agit par $\xi(g_l)$ et $\sigma \in W_v$ par $\Xi^{-w(\xi)/2}$.

Pour un $W_v$-faisceau pervers de Hecke $P$ sur $\bar X_\IC$,
$[H^*(P_\xi)]$
dÈsignera l'image de $\sum_i (-1)^i H^i(\bar X_\IC,P \otimes \LC_\xi)$ dans le groupe
de Grothendieck des reprÈsentations admissibles de $G(\Am^{\oo}) \times W_v$.
Pour $\groth$ le groupe de Grothendieck d'un groupe de la forme $G(\Am^{\oo,v}) \times \tilde
G$, on notera $\groth\{ \Pi^{\oo,v} \}$ le sous-groupe facteur direct de $\groth$ engendrÈ par les irrÈductibles de la forme $\Pi^{\oo,v} \otimes \sigma$ o˘ $\sigma$ est une reprÈsentation irrÈductible quelconque de $\tilde G$. On notera alors
$$[H^*(P_\xi)]\{ \Pi^{\oo,v} \}$$
la projection de $[H^*(P_\xi)]$ sur ce facteur direct.

\rem on rappelle que $G(\Qm_p) \simeq \Qm_p^\times \times
GL_d(F_v) \times \prod_{i=2}^r (B_{v_i}^{op})^\times$. Pour $\Pi$ une reprÈsentation de
$G(\Am)$, sa composante pour le facteur de similitude $\Qm_p^\times$, sera notÈe
comme dans \cite{h-t}, $\Pi_{p,0}$. Comme tous les compacts de $\IC$ contiennent le
facteur $\Zm_p^\times$, les reprÈsentations $\Pi$ qui vont intervenir par la suite, dans les diffÈrents groupes
de cohomologie, devront toutes vÈrifier que $(\Pi_{p,0})_{|\Zm_p^\times}=1$.

\subsection{Cohomologie des faisceaux pervers d'Harris-Taylor}

Soient
\begin{itemize}
\item $\Pi$ une reprÈsentation irrÈductible admissible de $G(\Am)$,
$m(\Pi)$ sa multiplicitÈ dans l'espace des formes automorphes et

\item $\pi_v$ une reprÈsentation irrÈductible cuspidale de $GL_g(F_v)$.
\end{itemize}
Nous allons rappeler d'aprËs \cite{boyer-compositio}, quels sont les couples $(r,i)$ tels que 
$$[H^i(\lexp p j^{\geq rg}_{!*} \FC_{\bar \Qm_l,\xi}(\pi_v,r)_{1})]\{ \Pi^{\oo,v} \} \neq (0).$$
On pourrait calculer explicitement l'image correspondante dans le groupe
de Grothendieck des $GL_{d-tg}(F_v) \times \Zm$-reprÈsentations mais nous nous 
contenterons d'une description trËs sommaire. Rappelons que
d'aprËs \cite{boyer-compositio} \S 3.6, cette description ne dÈpend
que de $m(\Pi)$, $\Pi_\oo$ et de la composante locale $\Pi_v$ de $\Pi$ en $v$.

\rem rappelons que l'action de $\sigma \in W_v$ sur 
$[H^i(\lexp p j^{\geq rg}_{!*} \FC_{\bar \Qm_l,\xi}(\pi_v,r)_{1})]\{ \Pi^{\oo,v} \}$
est donnÈe par celle de $\deg \sigma \in \Zm$ composÈe avec celle de 
$\Pi_{p,0}(\art^{-1} (\sigma))$ o˘ $\art^{-1}:W_v^{ab} \simeq F_v^\times$ est
l'isomorphisme d'Artin et $\deg$, l'application composÈe de $\art^{-1}$ avec la valuation.

\begin{lemm} \label{lem-compo-locale}
Soit $x$ une place de $\Qm$ dÈcomposÈe $x=yy^c$ dans $E$ et soit $z$
une place de $F$ au dessus de $y$ telle que, avec la notation \ref{nota-GFv},
$G(F_z) := (B_z^{op})^\times \simeq GL_d(F_z)$.
Pour $\Pi$ une reprÈsentation automorphe irrÈductible admissible cohomologique
de $G(\Am)$, sa composante locale $\Pi_z$, au sens de \ref{nota-GFv},
est de la forme $\speh_s(\pi_z)$ pour $\pi_z$ une reprÈsentation 
irrÈductible non dÈgÈnÈrÈe et $s$ un entier $\geq 1$ qui ne dÈpend
que de $\Pi$ et non de la place $z$ comme ci-dessus.
\end{lemm}

\begin{proof}
D'aprËs \cite{h-t} thÈorËme VI.2.1, une reprÈsentation automorphe de $G(\Am)$
est obtenue par changement de base ‡ partir d'une reprÈsentation automorphe de
$B^\times$, laquelle d'aprËs la correspondance de Langlands globale,
cf. \cite{badu} pour le cas gÈnÈral, est associÈe ‡ une sÈrie discrËte de $GL_d(\Am_F)$.
D'aprËs la classification des sÈries discrËtes de $GL_d(\Am_F)$ de \cite{m-w}, une
telle sÈrie discrËte est 
$$\speh_{s}(\pi):=\pi\{ \frac{1-s}{2} \} \boxplus \pi \{ \frac{3-s}{2} \} \boxplus \cdots \boxplus 
\pi \{ \frac{s-1}{2} \}$$ 
pour $\pi$ une reprÈsentation irrÈductible admissible cuspidale de $GL_{d/s}(\Am_F)$ et 
o˘ la notation $\boxplus$ correspond du cÙtÈ galoisien ‡ la somme directe.
La composante locale d'une telle reprÈsentation automorphe en une place $z$ est alors
$$\speh_{s}(\st_{t_1}(\pi_{1,z}))  \times \cdots \times \speh_{s}(\st_{t_u}(\pi_{u,z})) 
\simeq \speh_s \Bigl ( \st_{t_1}(\pi_{1,z}) \times \cdots \times \st_{t_u}(\pi_{u,z}) \Bigr )$$
o˘ les $\pi_{i,z}$ sont irrÈductibles cuspidales.
\end{proof}

\begin{coro} \label{coro-di} (cf. \cite{V-Z}) 
Soient 
\begin{itemize}
\item $x$ une place de $\Qm$ dÈcomposÈe
$x=yy^c$ dans $E$ et $z$ une place au dessus-de $y$ telle que $G(F_z) \simeq GL_d(F_z)$ et

\item $\Pi$ une reprÈsentation irrÈductible automorphe admissible de $G(\Am)$ telle que
sa composante locale $\Pi_z$ en $z$ soit de la forme $\speh_s(\pi_z)$ pour $\pi_z$
une reprÈsentation irrÈductible non dÈgÈnÈrÈe.
\end{itemize}
Alors les $d^i_\xi(\Pi_\oo)$ 
sont nuls pour $|i| \geq s$ ou pour $i \equiv s \mod 2$ et sinon ils sont tous Ègaux.
\end{coro}


\begin{nota} \label{nota-di}
Dans la suite nous noterons simplement $d_\xi(\Pi_\oo)$ pour la valeur commune non nulle
des $d^i_\xi(\Pi_\oo)$.
\end{nota}

\begin{defi} \label{defi-red}
Pour $1 \leq h \leq d$ et $\tau_v$ une reprÈsentation irrÈductible admissible de $D_{v,h}^\times$, on introduit d'aprËs 
\cite{h-t} V.5, 
$$\begin{array}{rccl}
R_{\tau_v}: & \groth \Bigl ( GL_{h}(F_v) \Bigr ) & \longto & \groth \Bigl (F_v^\times \Bigr ) \\
& \alpha & \mapsto & \mathrm{vol} (D_{v,tg}^\times/F_v^\times)^{-1} \sum_\psi \tr \alpha
(\varphi_{\tau_v \otimes \psi^\vee}) \psi ,
\end{array}$$
o˘ $\psi$ dÈcrit les caractËres de $F_v^\times$ et $\varphi_{\pi_v[t]_D}$ est, d'aprËs Deligne, Kazhdan et VignÈras,
un pseudo-coefficient pour $\st_t(\pi_v)$, ainsi que
$$\red_{\tau_v}: \groth \Bigl ( GL_d(F_v) \Bigr ) \longto \groth \Bigl ( F_v^\times \times GL_{d-h}
(F_v) \Bigr ) $$
dÈfini comme la composÈe de
$$\begin{array}{c}
\groth \Bigl ( GL_d(F_v) \Bigr ) \longto \groth \Bigl ( GL_{h}(F_v) \times GL_{d-h}(F_v) \Bigr ) \\
~ [ \Pi ] \mapsto [ J_{P_{h,d}^{op}}(\Pi) \otimes \delta_{P_{h,d}}^{1/2} ]
\end{array}$$
avec $R_{\tau_v} \otimes \Id$.
\end{defi}

\rem pour $\tau_v=\pi_v[t]_D$,
chacun des $\mathrm{vol} (D_{v,tg}^\times/F_v^\times)^{-1} \tr \alpha (\varphi_{\tau_v \otimes \psi^\vee})$
est nul si $\alpha$ n'est pas un sous-quotient irrÈductible de $\pi'_v \{ \frac{1-t}{2} \} \times
\cdots \times \pi'_v \{ \frac{t-1}{2} \}$ avec $\pi'_v \simeq \pi_v \otimes \psi \circ \det$. 
En ce qui concerne les sous-quotients irrÈductibles de cette induite,
ils sont d'aprËs Zelevinsky en bijection avec les orientations du graphe linÈaire dont les sommets
sont les entiers de $1$ ‡ $t$, les sommets reliant $k$ ‡ $k+1$ pour tout $k=1,\cdots,t-1$.
Pour un tel sous-quotient $\alpha$, 
$\mathrm{vol} (D_{v,tg}^\times/F_v^\times)^{-1} \tr \alpha (\varphi_{\tau_v \otimes \psi^\vee})$ 
est un signe qui ne dÈpend que du graphe et pas de $\pi'_v$, cf. par exemple la formule
donnÈe avant (1.5.3) dans \cite{boyer-compositio}.

\begin{prop} \label{prop-coho-gen} 
Soit $\Pi^{\oo,v}$ une reprÈsentation irrÈductible de $G(\Am^{\oo,v})$,
pour tout $(r,i)$, dans le groupe de Grothendieck des reprÈsentations
de $G(\Am^{\oo,v}) \times GL_{d-rg}(F_v) \times \Zm$, 
$[H^i(\lexp p j^{\geq rg}_{!*} \FC_{\bar \Qm_l,\xi}(\pi_v,r)_{1}[d-rg])]\{ \Pi^{\oo,v} \}$
est  Ègal ‡
$$e_{\pi_v} \Bigl ( \frac{\sharp \ker^1(\Qm,G)}{d} \sum_{\Pi' \in \UC_G(\Pi^{\oo,v})} 
m(\Pi') d_\xi(\Pi'_\oo) \Bigr ) R_{\pi_v}(r,i)(\Pi_v) $$
o˘  
\begin{itemize}
\item $\UC_G(\Pi^{\oo,v})$ dÈsigne l'ensemble des reprÈsentations 
irrÈductibles automorphes $\Pi'$ de $G(\Am)$ telles que $(\Pi')^{\oo,v} \simeq \Pi^{\oo,v}$;

\item $\Pi_v$ est la composante locale commune ‡ tous les $\Pi' \in \UC_G(\Pi^{\oo,v})$
tels que $d_\xi(\Pi'_\oo) \neq 0$, cf. le corollaire VI.2.2 de \cite{h-t};

\item $R_{\pi_v}(r,i)(\Pi_v)$ est une somme de reprÈsentations de $GL_{d-rg}(F_v) \times \Zm$ 
dÈfinie, pour 
$\Pi_v \simeq  \speh_s(\st_{t_1}(\pi_{1,v})) \times \cdots \times  \speh_s(\st_{t_u}(\pi_{u,v}))$,  
o˘ les $\pi_{i,v}$ sont des reprÈsentations irrÈductibles cuspidales de $GL_{g_i}(F_v)$,
par la formule
$$R_{\pi_v}(r,i)(\Pi_v)=\sum_{k:~\pi_{k,v} \sim_i \pi_{v}} 
m_{s,t_k}(r,i) R_{\pi_{v}}(s,t_k)(r,i)(\Pi_v,k) \otimes \Bigl ( \xi_k \otimes \Xi^{i/2} \Bigr )$$ o˘:
\begin{itemize}
\item les $\xi_k$ sont tels que $\pi_{k,v} \simeq \pi_v \otimes \xi_k \circ \val \circ \det$;

\item $R_{\pi_{v}}(s,t_k)(r,i)(\Pi_v,k) $ s'Ècrit comme suit
\addtocounter{smfthm}{1}
\begin{multline*} 
R_{\pi_{v}}(s,t_k)(r,i)(\Pi_v,k) :=\speh_s(\st_{t_1}(\pi_{1,v})) \times \cdots \times  \speh_s(\st_{t_{k-1}}(\pi_{k-1,v})) \\ 
\times R_{\pi_{k,v}}(s,t_k)(r,i) \times \\ 
\speh_s(\st_{t_{k+1}}(\pi_{k+1,v})) \times \cdots \times  \speh_s(\st_{t_u}(\pi_{u,v})).
\end{multline*}

\item $R_{\pi_{k,v}}(s,t_k)(r,i)$ est une reprÈsentation de $GL_{d-rg}(F_v)$ qui s'obtient comme expliquÈ ci-aprËs.

\item $m_{s,t}(r,i) \in \{ 0, 1 \}$ est dÈfini ci-aprËs.
\end{itemize}
\end{itemize}
\end{prop}

La reprÈsentation $R_{\pi_{k,v}}(s,t_k)(r,i)$ est calculÈe comme suit: on applique le foncteur de Jacquet
$J_{P^{op}_{rg,d}}$ ‡ $\speh_s(\st_{t_k}(\pi_{k,v}))$ qui s'Ècrit, ‡ la Zelevinsky, comme une somme 
$$\sum \langle a_1 \rangle \otimes \langle a_2 \rangle$$ 
o˘ $a_1,a_2$ sont des multisegments dans la droite de Zelevinsky de $\pi_v$; le calcul explicite est donnÈ 
au corollaire 1.5.6 de \cite{boyer-compositio} avec la notation 1.5.7. On considËre alors la somme
$$\sum R_{\pi_v[r]_D} \Bigl ( \langle a_1 \rangle \Bigr ) a_2=\sum_\psi  \epsilon_\psi \psi \otimes \Pi_\psi \in 
\groth \Bigl ( F_v^\times \times GL_{(st-r)g}(F_v) \Bigr )$$
o˘ $\psi$ dÈcrit les caractËres de $F_v^\times$ et $\epsilon_\psi \in \{ -1,1 \}$.

\rem dans la formule prÈcÈdente, o˘ $\tau_v=\pi_v[r]_D$, 
les $\psi$ tels que $\Pi_\psi$ sont non nuls, sont de la forme $|-|^{k/2}$ avec $k \in \Zm$.

D'aprËs \cite{h-t} thÈorËme V.6.1, cette formule permet de calculer la somme
alternÈe des $[H^i(\lexp p j^{\geq rg}_{!} \FC_{\bar \Qm_l,\xi}(\pi_v,r)_{1}[d-rg])]\{ \Pi^{\oo,v} \}$.
Partant de l'ÈgalitÈ, cf. \cite{boyer-compositio} 2.6.3
\addtocounter{smfthm}{1}
\begin{equation} \label{eq-se-2}
i^{tg}_{\IC,*} j^{\geq tg}_{\IC,!*} HT(\pi_v,\Pi_t)= \sum_{r=0}^{s-t} (-1)^r i^{(t+r)g}_{\IC,*} j^{\geq (t+r)g}_{\IC,!}
HT(\pi_v,\Pi_t \overrightarrow{\times} [\overrightarrow{r-1}]_{\pi_v}) (r/2)
\end{equation}
par puretÈ, le calcul de chacun des $H^i (j^{\geq tg}_{\IC,!*} HT(\pi_v,\Pi_t))$ se dÈduit du calcul de la somme
alternÈe des groupes de cohomologie de tous les $j^{\geq (t+r)g}_{\IC,!}
HT(\pi_v,\Pi_t \overrightarrow{\times} [\overrightarrow{r-1}]_{\pi_v}) (r/2)$.

D'aprËs le lemme 1.5.5 de 
\cite{boyer-compositio}, les applications $\red_{\tau_v}$ sont multiplicatives au sens o˘
$$\red_{\tau_v} \Bigl ( \pi_1 \times \cdots \times \pi_r \Bigr ) = \sum_{i=1}^r \pi_1 \times \cdots \times \pi_{i-1} \times
( \red_{\tau_v} \pi_i ) \times \pi_{i+1} \times \cdots \times \pi_r.$$
Ainsi, par \og superposition \fg, d'aprËs la proposition 3.6.1 de \cite{boyer-compositio}, on a 
$$R_{\pi_{k,v}}(s,t_k)(r,i)=\Pi_{|-|^{-i/2}}.$$
Les conditions sur $i$ dans la proposition 3.6.1 de \cite{boyer-compositio} sont intÈgrÈes dans les
coefficients $m_{s,t}(r,i)$ dÈcrits ci-aprËs.

\rem notons que
pour $\pi_v$ irrÈductible cuspidale entiËre, il rÈsulte de la description prÈcÈdente que $R_{\pi_v}(s,t)(r,i)$
est entiËre et sa rÈduction modulo $l$ ne dÈpend que $(s,t,r,i)$ et de $r_l(\pi_v)$. 
En effet considÈrons deux reprÈsentations entiËres irrÈductibles cuspidales $\pi_v$ et $\pi'_v$ de mÍme rÈduction 
modulo $l$.
Le foncteur de Jacquet $J_{P^{op}_{rg,d}}$ appliquÈ ‡ $R_{\pi_v}(s,t)(r,i)$ s'Ècrit comme prÈcÈdemment
$\sum \langle a_1 \rangle  \otimes \langle a_2 \rangle$ o˘ $a_1,a_2$ sont des multisegments dans la droite de 
Zelevinsky de $\pi_v$. Si dans chacun de ces multisegments, on remplace les $\pi_v \{ \cdots \}$ par les 
$\pi'_v \{ \cdots \}$, le rÈsultat obtenu est
$$\sum \langle a'_1 \rangle \otimes \langle a'_2 \rangle =  J_{P^{op}_{rg,d}} \bigl ( R_{\pi'_v}(s,t)(r,i) \bigr ) .$$
Le rÈsultat dÈcoule alors de la description explicite de 
$R_{\pi_v[r]_D}(\langle a_1 \rangle )= \pm |-|^{k/2}$ o˘, d'aprËs la remarque prÈcÈdant la proposition \ref{prop-coho-gen},
le signe et l'entier $k$ ne dÈpendent que du graphe associÈ au multisegment $a_1$ et pas de $\pi_v$.

\begin{defi} Les points de coordonnÈes $(r,i)$ tels que $m_{s,t}(r,i)=1$ sont contenus dans l'enveloppe convexe 
du polygone de sommets $(s+t-1,0)$, $(t,\pm (s-1))$ et $(1, \pm (s-t))$ si $s \geq t$ (resp.
$(t-s+1,0)$ si $t \geq s$); ‡ l'intÈrieur de celui-ci pour $r$ fixÈ, les $i$ concernÈs partent du 
bord en vont de $2$ en $2$, i.e. $m_{s,t}(r,i)=1$ si et seulement si:
\begin{itemize}
\item $\max†\{ 1, s+t-1-2(s-1) \}  \leq r \leq s+t-1$;

\item si $t \leq r \leq s+t-1$ alors $0 \leq |i| \leq s+t-1-r$ et $i \equiv s+t-1-r \mod 2$;

\item si $\max \{ 1,s+t-1-2(s-1) \} \leq r \leq t$ alors $0 \leq |i| \leq s-1-(t-r)$ et
$i \equiv s-t-1+r \mod 2$.
\end{itemize}
Le lecteur trouvera des illustrations de cette dÈfinition aux figures 
\ref{fig-coho-m1} et \ref{fig-coho-m2}.
\end{defi}

\begin{figure}[ht]
\centering
\includegraphics{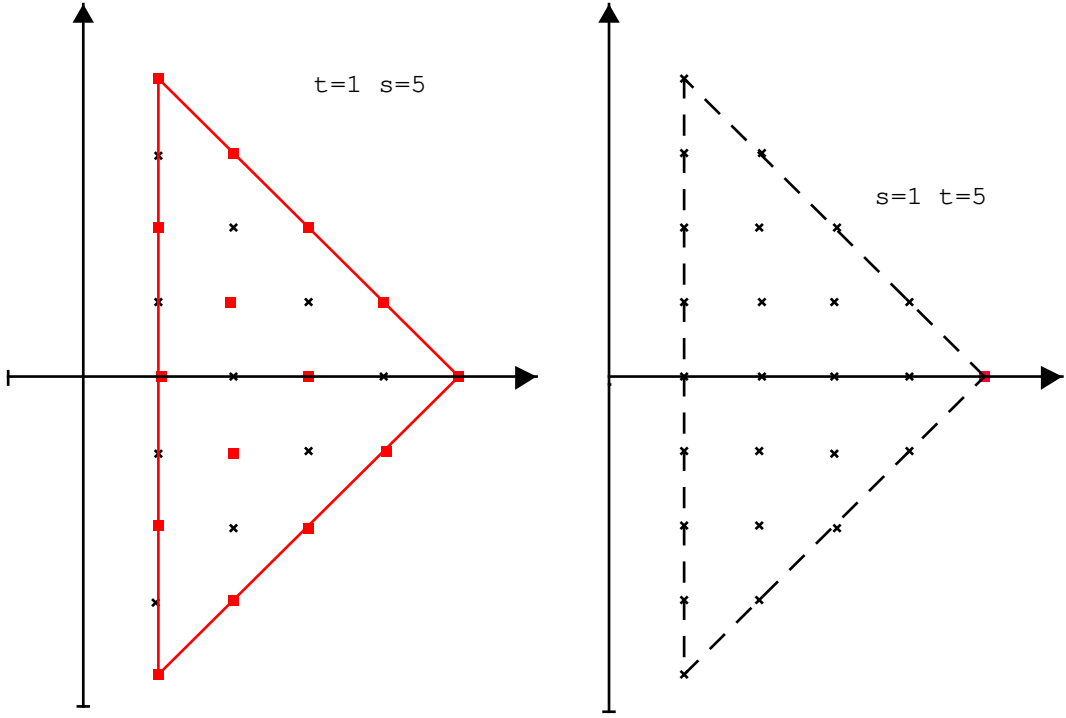}
\caption{\label{fig-coho-m1} Les carrÈs rouges
donnent les coordonnÈes $(r,i)$ des
$m_{s,t}(r,i)=1$ dans le cas d'une Speh ($t=1$) ‡ gauche et d'une Steinberg ($s=1$) ‡ droite}
\end{figure}

\begin{figure}[ht]
\centering
\includegraphics{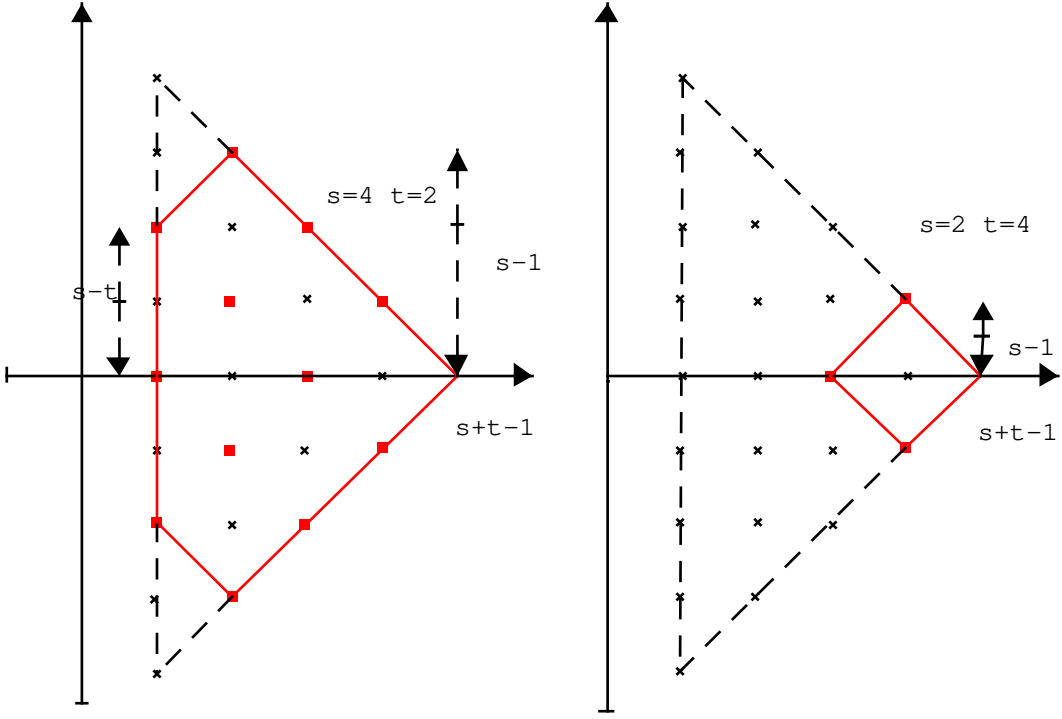}
\caption{\label{fig-coho-m2} Position des $m_{s,t}(r,i)=1$ dans le cas $s \geq t$ ‡ gauche et 
$t \geq s$ ‡ droite}
\end{figure}


\begin{figure}[ht]
\centering
\includegraphics[scale=.8]{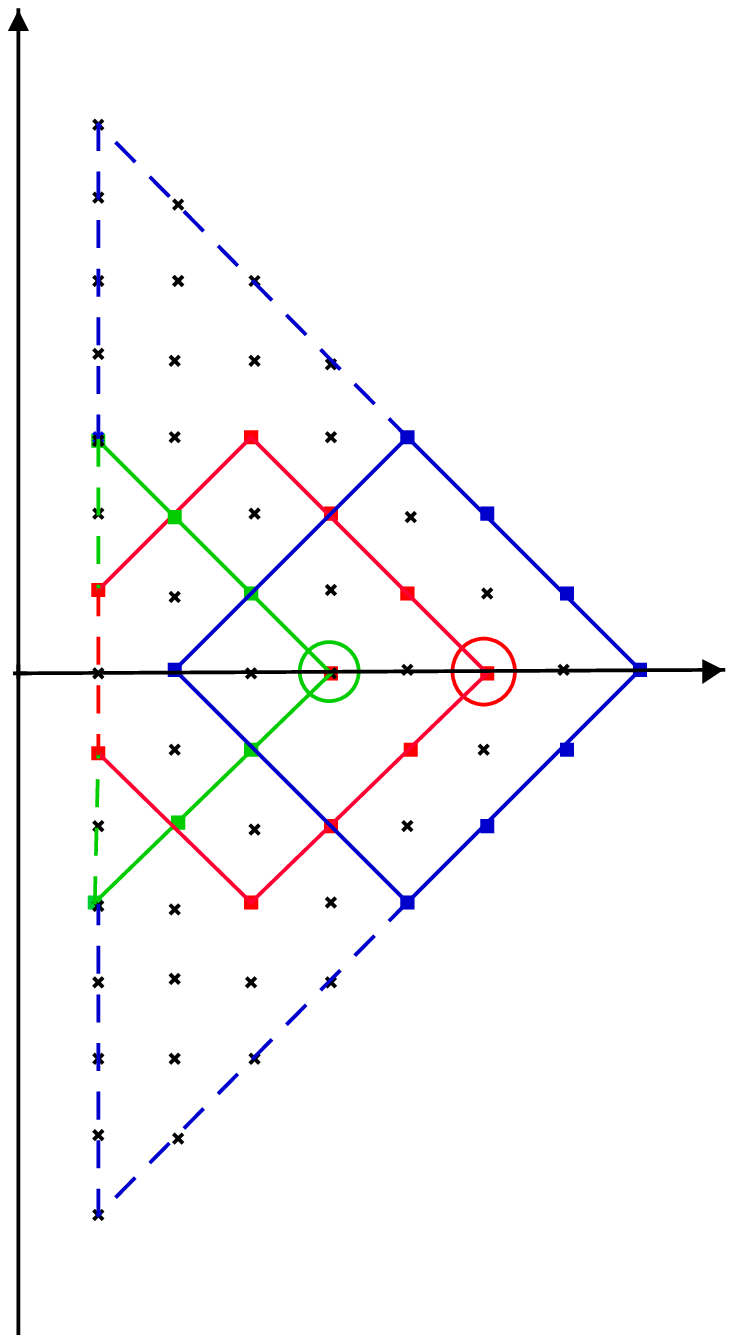}
\caption{\label{fig-coho-m4} Superposition des diagrammes pour le calcul des $m(r,i)$ dans le cas 
o˘ $\Pi_v \simeq \speh_4(\pi_v) \times \speh_4(\st_3(\pi_v)) \times \speh_4(\st_5(\pi_v))$}
\end{figure}

\rem pour
$$\Pi_v \simeq \speh_s(\st_{t_1}(\pi_{1,v})) \times \cdots \times \speh_s(\st_{t_u}(\pi_{u,v}))$$
l'ensemble des couples $(r,i)$ tels que $[H^i(\lexp p j^{\geq rg}_{!*} \FC_{\bar \Qm_l,\xi}(\pi_v,r)[d-rg])]\{ \Pi^{\oo,v} \}$
est obtenu en superposant les $u$ diagrammes prÈcÈdents comme sur la figure \ref{fig-coho-m4}; pour $(r,i)$
donnÈ la contribution du diagramme d'indice $k$ s'obtient en remontant ‡ la source, i.e. ‡
$(s+t_{k}-1,0)$ et en rempla\c cant $\speh_s(\st_{t_k}(\pi_{k,v}))$ par $R_{\pi_{k}}(s,t_{k})(r,i)$. Ainsi:
\begin{itemize}
\item pour tout $i \neq 0$ et pour tout $r$ tel que 
$[H^i(\lexp p j^{\geq rg}_{!*} \FC_{\bar \Qm_l,\xi}(\pi_v,r)_1[d-rg])]\{ \Pi^{\oo,v} \}$ 
est non nulle alors il existe $r'>r$ tel que 
$[H^0(\lexp p j^{\geq r'g}_{!*} \FC_{\bar \Qm_l,\xi}(\pi_v,r')_1[d-rg])]\{ \Pi^{\oo,v} \}$ est non 
nulle. Plus prÈcisÈment tout
$R_{\pi_k}(s,t_k)(r,i)(\Pi'_v)\otimes \Bigl ( \xi_k \otimes \Xi^{i/2} \Bigr )$ de $[H^i(\lexp p j^{\geq rg}_{!*} \FC_{\bar \Qm_l,\xi}(\pi_v,r)_1[d-rg])]\{ \Pi^{\oo,v} \}$, est associÈ ‡ un 
$R_{\pi_k}(s+t_k,0)(r',0)(\Pi'_v) \otimes \Bigl ( \xi_k \otimes \Xi^{0} \Bigr )$ de 
$[H^0(\lexp p j^{\geq (s+t_k)g}_{!*} \FC_{\bar \Qm_l,\xi}(\pi_v,s+t_k)_1[d-(s+t_k)g])]\{ \Pi^{\oo,v} \}$.

\item pour $i=0$, certains des constituants de 
$[H^0(\lexp p j^{\geq rg}_{!*} \FC_{\bar \Qm_l,\xi}(\pi_v,r)_1[d-rg])]\{ \Pi^{\oo,v} \}$ \og proviennent \fg{} des $r' >r$
comme dans le cas prÈcÈdent et d'autres non.
\end{itemize}
\noindent \textit{Illustration sur la figure \ref{fig-coho-m4} pour $r=4$}: 
\begin{itemize}
\item $ \speh_4(\pi_v) \times \speh_4(\st_3(\pi_v)) \times R_{\pi_v}(4,5)(4,0)$ provient de $(8,0)$;

\item $ \speh_4(\pi_v) \times  R_{\pi_v}(4,3)(4,0) \times \speh_4(\st_5(\pi_v))$ provient de $(6,0)$;

\item $R_{\pi_v}(4,1)(4,0) \times \speh_4(\st_3(\pi_v)) \times \speh_4(\st_5(\pi_v))$ ne provient pas d'un $(r',0)$ pour $r'>4$.
\end{itemize}

\begin{nota} \label{nota-AC}
Pour $\pi_v$ une reprÈsentation irrÈductible cuspidale de $GL_g(F_v)$, on note
$\AC_{\xi,\pi_v}(r,s)$ l'ensemble des reprÈsentations irrÈductibles automorphes de $G(\Am)$
$\xi$-cohomologique telles que:
\begin{itemize}
\item $(\Pi_{p,0})_{|\Zm_p^\times}=1$,

\item sa composante locale en $v$ soit de la forme $\speh_s(\st_t(\pi'_v)) \times ?$
avec $r=s+t-1$, $\pi'_v \sim^i \pi_v$ et $?$ une reprÈsentation quelconque de $GL_{d-stg}(F_v)$.
\end{itemize}
\end{nota}

\rem ainsi une reprÈsentation irrÈductible automorphe de $G(\Am)$ $\xi$-cohomologique
dont la composante locale est de la forme
$$ \speh_s \Bigl (\st_{t_1}(\pi_{1,v}) \times \cdots \times \st_{t_u}(\pi_{u,v}) \Bigr )= \speh_s(\st_{t_1}(\pi_{1,v})) \times
\cdots \times \speh_s(\st_{t_u}(\pi_{u,v}))$$
appartient ‡ $\AC_{\xi,\pi_{v,i}}(s+t_i-1,s)$ pour $i=1,\cdots,u$. En particulier ces ensembles
$\AC_{\xi,\pi_v}(r,s)$ ne sont pas disjoints.

\begin{coro} 
Soit $\pi_v$ une reprÈsentation irrÈductible cuspidale de $GL_g(F_v)$; pour tout
$1 \leq r \leq d/g$, on a
\begin{multline*}
[H^i(\lexp p j^{\geq rg}_{!*} \FC_{\bar \Qm_l,\xi}(\pi_v,r)_{1}[d-rg])] = 
\frac{e_{\pi_v} \sharp \ker^1(\Qm,G)}{d} \\
\sum_{\atop{(s,t)}{m_{s,t}(r,i)=1}} \sum_{\Pi \in \AC_{\xi,\pi_v}(s+t-1,s)} m(\Pi) d_\xi(\Pi_\oo) 
\Bigl ( \Pi^{\oo,v}\otimes R_{\pi_v}(s,t)(r,i)(\Pi_v) \Bigr )
\end{multline*}
o˘ avec les notations prÈcÈdentes, $R_{\pi_v}(s,t)(r,i)(\Pi_v)$ est donnÈ pour
$\Pi_v \simeq  \speh_s(\st_{t_1}(\pi_{1,v})) \times \cdots \times  \speh_s(\st_{t_u}(\pi_{u,v}))$,  
par la formule
$$\sum_{\atop{k:~\pi_{k,v} \sim_i \pi_{v}}{t_k=t}} 
R_{\pi_{v}}(s,t)(r,i)(\Pi_v,k) \otimes \Bigl ( \xi_k \otimes \Xi^{i/2} \Bigr ),$$ 
avec $\pi_{k,v} \simeq \pi_v \otimes \xi_k \circ \val \circ \det$.
\end{coro}

\subsection{Cohomologie des systËmes locaux d'Harris-Taylor}
\label{para-coho-c}

On reprend les calculs du paragraphe prÈcÈdent pour les 
$j^{\geq rg}_! \FC_{\Qm_{l},\xi}(\pi_{v},r)$.  Rappelons que la somme alternÈe de 
leurs groupes de cohomologie est donnÈe dans \cite{h-t} ce qui a permis, cf. le 
paragraphe prÈcÈdent, de calculer chacun des groupes de cohomologie des extensions
intermÈdiaires. En utilisant l'ÈgalitÈ suivante, cf. \cite{boyer-invent2} corollaire 5.4.1
\addtocounter{smfthm}{1}
\begin{multline} \label{egalite-hij}
i_{\IC,*}^{tg} j^{\geq tg}_{!} HT(\pi_v,\Pi_t)= i_{*}^{tg} j^{\geq tg}_{!*} HT(\pi_v,\Pi_t) + \\
\sum_{i=1}^{s-t} i_{*}^{(t+i)g} j^{\geq (t+i)g}_{!*} HT(\pi_v,\Pi_t \overrightarrow{\times}
\st_i(\pi_v))(i/2)
\end{multline}
on en dÈduit le calcul de chacun des $H^i (j^{\geq tg}_{\IC,!} HT(\pi_v,\Pi_t))$. 
Comme prÈcÈdemment on traite tout d'abord le cas de $\speh_s(\st_t(\pi_v))$.

\begin{prop} (cf. \cite{boyer-compositio} \S 5) \label{prop-coho-n0} \\
Dans le cas o˘ il existe $\Pi' \in \UC_G(\Pi^{\oo,v})$ avec $d_\xi(\Pi'_\oo) \neq 0$ et
$\Pi'_v \simeq \speh_s(\st_t(\pi_v))$ avec $\pi_v$ 
une reprÈsentation irrÈductible cuspidale de  $GL_g(F_v)$, alors pour tout $(r,i)$,
$[H^i(\lexp p j^{\geq rg'}_{!} \FC_{\bar \Qm_l,\xi}(\pi'_v,r)_1)]\{ \Pi^{\oo,v} \}$ est nul si $g' \neq g$ et $\pi'_v$
n'est pas dans la classe inertielle de $\pi_v$ et sinon Ègal ‡
$$n_{s,t}(r,i)e_{\pi_v} \Bigl ( \frac{\sharp \ker^1(\Qm,G)}{d} 
\sum_{\Pi' \in \UC_G(\Pi^{\oo,v})}  m(\Pi') d_\xi(\Pi'_\oo) \Bigr )
S_{\pi_v}(s,t)(r,i) \otimes \Bigl ( \xi \otimes \Xi^{\frac{2i+r-s-t+1}{2}} \Bigr ),$$
o˘ $\pi'_v \simeq \pi_v \otimes \xi \circ \val \circ \det$ et
$S_{\pi_v}(s,t)(r,i)$ est une reprÈsentation de $GL_{d-tg}(F_v)$ qui s'obtient 
combinatoirement ‡ partir de $\speh_s(\st_t(\pi_v))$ 
et $n_{s,t}(r,i) \in \{ 0, 1 \}$  est dÈfini ci-aprËs.
\end{prop}

\begin{defi}
Les points de coordonnÈes $(r,i)$ tels que $n_{s,t}(r,i)=1$ sont tous 
ceux contenus dans l'enveloppe convexe 
du polygone de sommets $(s+t-1,0)$, $(s,0)$, $(1, s-1)$ et $(t,s-1)$.
\end{defi}

\rem on trouvera des illustrations de cette dÈfinition aux figures 
\ref{fig-coho-n1} et \ref{fig-coho-n2}.

\rem ‡ vrai dire dans \cite{boyer-compositio} on traite seulement les cas $s=1$ ou $t=1$. Le calcul
procËde par analyse de la suite spectrale associÈe ‡ la filtration de stratification de
$ j^{\geq tg}_{\IC,!} HT(\pi_v,\Pi_t)$ dont les graduÈs sont les termes du membre de droite de
(\ref{egalite-hij}). La nullitÈ des $n_{s,t}(r,i)$ pour $(r,i)$ n'appartenant pas ‡ l'intÈrieur du quadrilatËre
de sommets $(1,0)$, $(1,s-1)$, $(t,s-1)$ et $(s+t-1,0)$ est alors immÈdiate (en utilisant la nullitÈ
pour $i<0$). De mÍme on obtient $n_{s,t}(r,i)$ pour $(r,i)$ appartenant ‡ la ligne polygonale joignant
$(1,s-1)$, $(t,s-1)$ et $(s+t-1,0)$ et que ceux qui sont non nuls sont tels que $i$ a la mÍme paritÈ que
le $(r,i')$ sur cette ligne polygonale. Le seul argument combinatoire rÈside alors dans l'Ètude de
la nullitÈ pour $(r,i)$ appartenant ‡ l'intÈrieur de $(1,s-1)$, $(s,0)$, et $(1,0)$. Dans la suite nous utiliserons
simplement le fait que $S_{\pi_v}(s,t)(s+t-1,0)=R_{\pi_v}(s,t)(s+t-1,0)$ qui dÈcoule directement
de ce que nous venons d'expliciter.

\begin{figure}[ht]
\centering
\includegraphics{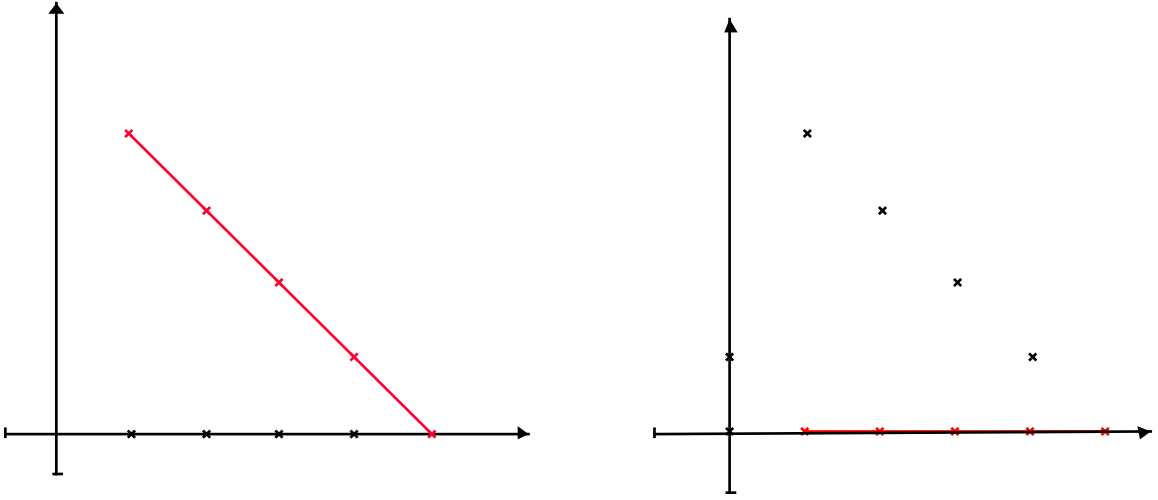}
\caption{\label{fig-coho-n1} Position des $n_{s,t}(r,i)=1$ 
dans le cas d'une Speh ‡ gauche et d'une Steinberg ‡ droite}
\end{figure}

\begin{figure}[ht]
\centering
\includegraphics{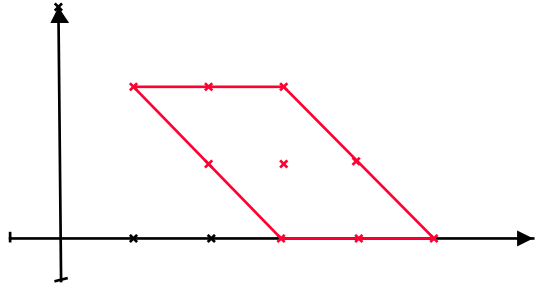}
\caption{\label{fig-coho-n2} Position des $n_{s,t}(r,i)=1$ pour $\speh_3(\st_3(\pi_v))$}
\end{figure}

Comme dans le paragraphe prÈcÈdent la multiplicativitÈ de $\red$, cf.
le lemme 1.5.5 de \cite{boyer-compositio}, permet d'obtenir le cas gÈnÈral.

\begin{prop} \label{prop-coho-gen2} 
Soit $\Pi^{\oo,v}$ une reprÈsentation irrÈductible de $G(\Am^{\oo,v})$,
pour tout $(r,i)$, dans le groupe de Grothendieck des reprÈsentations
de $G(\Am^{\oo,v}) \times GL_{d-rg}(F_v) \times \Zm$, 
$[H^i(\lexp p j^{\geq rg}_{!} \FC_{\bar \Qm_l,\xi}(\pi_v,r)_{1}[d-rg])]\{ \Pi^{\oo,v} \}$
est  Ègal ‡
$$n_{s,t}(r,i)e_{\pi_v} \Bigl ( \frac{\sharp \ker^1(\Qm,G)}{d} \sum_{\Pi' \in \UC_G(\Pi^{\oo,v})} 
m(\Pi') d_\xi(\Pi'_\oo) \Bigr ) S_{\pi_v}(r,i)(\Pi_v) $$
o˘ 
\begin{itemize}
\item $\Pi_v$ est la composante locale commune ‡ tous les $\Pi' \in \UC_G(\Pi^{\oo,v})$
tels que $d_\xi(\Pi'_\oo) \neq 0$, cf. le corollaire VI.2.2 de \cite{h-t};

\item $S_{\pi_v}(r,i)(\Pi_v)$ est une somme de reprÈsentations de $GL_{d-rg}(F_v) \times \Zm$ 
dÈfinie, pour 
$\Pi'_v \simeq  \speh_s(\st_{t_1}(\pi_{1,v})) \times \cdots \times  \speh_s(\st_{t_u}(\pi_{u,v}))$,  
o˘ les $\pi_{i,v}$ sont des reprÈsentations irrÈductibles cuspidales de $GL_{g_i}(F_v)$
par la formule
$$S_{\pi_v}(r,i)(\Pi_v)=\sum_{k:~\pi_{k,v} \sim_i \pi_{v}} 
n_{s,t_k}(r,i) S_{\pi_{v}}(s,t_k)(r,i)(\Pi_v,k) \otimes \Bigl ( \xi_k \otimes \Xi^{\frac{2i+r-s-t+1}{2}} \Bigr )$$ o˘:
\begin{itemize}
\item les $\xi_k$ sont tels que $\pi_{k,v} \simeq \pi_v \otimes \xi_k \circ \val \circ \det$;

\item $S_{\pi_{v}}(s,t_k)(r,i)(\Pi_v,k) $ s'Ècrit comme suit
\addtocounter{smfthm}{1}
\begin{multline*} 
S_{\pi_{v}}(s,t_k)(r,i)(\Pi_v,k) :=\speh_s(\st_{t_1}(\pi_{1,v})) \times \cdots \times  \speh_s(\st_{t_{k-1}}(\pi_{k-1,v})) \\ 
\times S_{\pi_{k,v}}(s,t_k)(r,i) \times \\ 
\speh_s(\st_{t_{k+1}}(\pi_{k+1,v})) \times \cdots \times  \speh_s(\st_{t_u}(\pi_{u,v})).
\end{multline*}
\end{itemize}
\end{itemize}
\end{prop}

\begin{figure}[ht]
\centering
\includegraphics{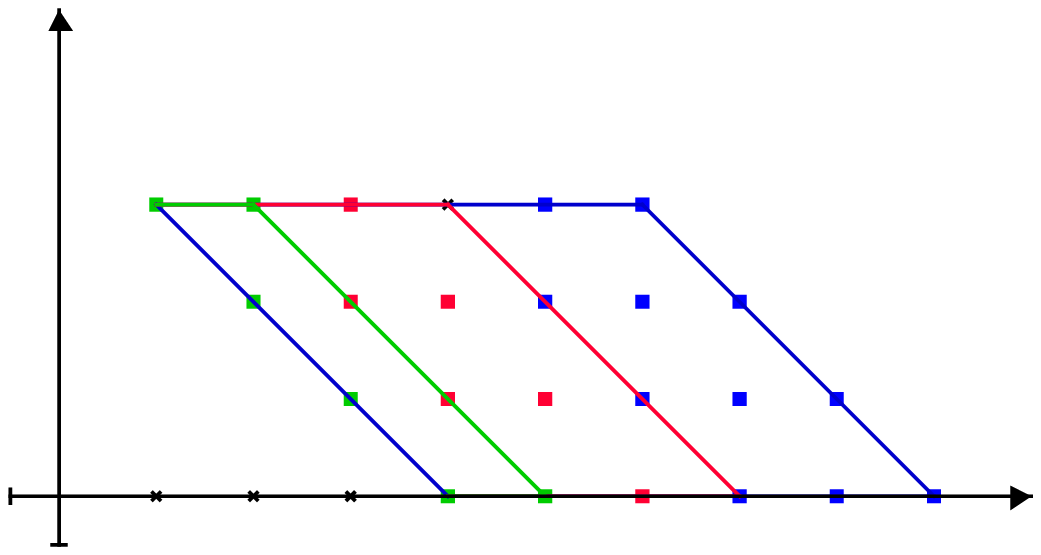}
\caption{\label{fig-coho-n4} Superposition des diagrammes pour le calcul des $n(r,i)$ dans le cas 
o˘ $\Pi_v \simeq  \speh_4(\pi_v) \times \speh_4(\st_3(\pi_v)) \times \speh_4(\st_5(\pi_v))$}
\end{figure}

\begin{coro} \label{coro-coho0}
Soit $\pi_v$ une reprÈsentation irrÈductible cuspidale de $GL_g(F_v)$; pour tout
$1 \leq r \leq d/g$, on a
\begin{multline*}
[H^i(\lexp p j^{\geq rg}_{!} \FC_{\bar \Qm_l,\xi}(\pi_v,r)_{1}[d-rg])] = 
\frac{e_{\pi_v} \sharp \ker^1(\Qm,G)}{d} \\
\sum_{\atop{(s,t)}{n_{s,t}(r,i)=1}} \sum_{\Pi \in \AC_{\xi,\pi_v}(s+t-1,s)} m(\Pi) d_\xi(\Pi_\oo) 
\Bigl ( \Pi^{\oo,v} \otimes S_{\pi_v}(s,t)(r,i)(\Pi_v) \Bigr )
\end{multline*}
o˘ avec les notations prÈcÈdentes, $S_{\pi_v}(s,t)(r,i)(\Pi_v)$ est donnÈ pour
$\Pi_v \simeq  \speh_s(\st_{t_1}(\pi_{1,v})) \times \cdots \times  \speh_s(\st_{t_u}(\pi_{u,v}))$,  
par la formule
$$\sum_{\atop{k:~\pi_{k,v} \sim_i \pi_{v}}{t_k=t}} 
S_{\pi_{v}}(s,t)(r,i)(\Pi_v,k) \otimes \Bigl ( \xi_k \otimes \Xi^{\frac{2i+r-s-t+1}{2}} \Bigr ),$$ 
avec $\pi_{k,v} \simeq \pi_v \otimes \xi_k \circ \val \circ \det$.
\end{coro}

\section{Augmentation de l'irrÈductibilitÈ}
\label{para-congru2}

\subsection{Congruences faibles entre reprÈsentations automorphes}
\label{para-formulation}

CommenÁons par rappeler la notion classique de congruence faible 
entre reprÈsentations automorphes de $G(\Am)$ comme le lecteur pourra le trouver
dans le \S 3 de \cite{vigneras-langlands}. 
Pour un corps
local non archimÈdien $K$ d'anneau des entiers $\OC_K$, et un groupe rÈductif $G$ connexe
non ramifiÈ sur $K$ tel que $G(\OC_K)$ est dÈfini, Ègal ‡ un sous-groupe ouvert compact
maximal spÈcial de $G(K)$, la $\Zm$-algËbre de Hecke sphÈrique $\HC_K$ de $G(K)$
par rapport ‡ $G(\OC_K)$ est commutative.

Pour $\pi_K$ une $\bar \Qm_l$-reprÈsentation
irrÈductible non ramifiÈe entiËre sur l'anneau des entiers $\OC$ d'une extension finie de
$\Qm_l$ contenue dans $\bar \Qm_l$, l'action de $\HC_K$ sur l'espace vectoriel
de dimension $1$, $\pi_K^{G(\OC_K)}$, est donnÈe par le caractËre de Satake
$$\lambda_{\pi_K}:\HC_K \longrightarrow \OC.$$

\begin{defi} 
Deux $\bar \Qm_l$-reprÈsentations $\pi_K$ et $\pi'_K$, irrÈductibles, $\OC$-entiËres et non
ramifiÈes de $G(K)$ sont dites \emph{congruentes modulo $l$} si leurs caractËres
de Satake $\lambda_{\pi_K}$ et $\lambda_{\pi'_K}$ ont la mÍme rÈduction modulo $l$.
\end{defi}

Dans le cas global, soient $\PC$ l'ensemble des places finies de $\Qm$ et $T \subset \PC$
de complÈmentaire fini tel que pour tout $w \in T$, $G(\Qm_w)$ est non ramifiÈ et
$G(\Zm_w)$ est dÈfini Ègal ‡ un sous-groupe compact ouvert spÈcial de $G(\Qm_w)$.
Pour $\Pi$ une $\bar \Qm_l$-reprÈsentation irrÈductible $\OC$-entiËre de $G(\Am)$ 
non ramifiÈe en toute place de $T$, la $\Zm$-algËbre de Hecke sphÈrique $\HC_T$ 
de $G(\Qm_T):=\prod_{w \in T} G(\Qm_w)$ relativement ‡ $G(\Zm_T):=
\prod_{w \in T} G(\Zm_w)$, agit sur la droite $\pi_T^{G(\Zm_T)}$ via le
caractËre de Satake de $\Pi_T$:
$$\lambda_{T,\Pi}:\HC_T \longrightarrow \OC.$$

\begin{defi} \label{defi-faible}
Deux $\bar \Qm_l$-reprÈsentations $\Pi$ et $\Pi'$, irrÈductibles, $\OC$-entiËres et non
ramifiÈes de $G(\Am)$ sont dites \emph{$T$-congruentes modulo $l$} si leurs caractËres
de Satake $\lambda_{T,\pi}$ et $\lambda_{T,\Pi'}$ ont la mÍme rÈduction modulo $l$.
Elles sont dites \emph{faiblement congruentes}, s'il existe un tel ensemble de places $T$
telles qu'elles soient $T$-congruentes.
\end{defi}

\rem $\Pi$ et $\Pi'$ sont $T$-congruentes modulo $l$ si et seulement si 
$\Pi_w$ et  $\Pi'_w$ sont congruentes pour toute place $w$ de $T$.

Pour tout ensemble $T$ comme ci-avant et pour toute reprÈsentation
irrÈductible $\Pi$ de $G(\Am)$ non ramifiÈe en $T$, on considËre 
$\Pi_T^{G(\Zm_T)}$ comme un $\HC_T$-module et $\Pi^{T,\oo}$ comme une
reprÈsentation de $\prod_{w \not \in T \cup \{ \oo \} } G(\Qm_w)$. 

\begin{nota} 
Pour tout ensemble $T$ de places finies $w$ de $\Qm$ comme prÈcÈdemment, avec
$p \not \in T$, on notera $\GF_{\bar \Qm_l,T}(v,h)$ le groupe de Grothendieck des 
$\bar \Qm_l$-reprÈsentations de longueur finie de 
$\prod_{w \not \in T \cup \{ \oo,p \} } G(\Qm_w) \times \prod_{i=2}^r G(F_{v_i}) \times 
GL_h(F_v) \times \Zm$ munies d'une action de $\HC_T$.
Pour $\Pi_T \in \GF_{\bar \Qm_l,T}(v,h)$ et $T' \subset T$, $\Pi_T^{G(\Zm_{T-T'})}$ est un objet
de $\GF_{\bar \Qm_l,T'}(v,h)$ et on note $\GF_{\overline{\Qm}_l}(v,h)$ la limite projective associÈe.
\end{nota}

\rem si $\Pi^{\oo,v} \otimes \Pi'_v \otimes \chi$ est une reprÈsentation entiËre de 
$G(\Am^{\oo,v}) \times GL_h(F_v) \times \Zm$, alors, cf. l'appendice A de  
\cite{vigneras-langlands}, ses caractËres de Satake et chacune de ses composantes $\Pi_w$ 
sont aussi entiers.

On voit la rÈduction modulo $l$ de $\Pi_T^{G(\Zm_T)} \otimes \Pi^{T,\oo} \otimes \Pi'_v 
\otimes \chi$ comme un objet du groupe de Grothendieck 
$\GF_{\bar \Fm_l,T}(v,h)$ des $\bar \Fm_l$-reprÈsentations de longueur finie de 
$GL_h(F_v) \times \Zm \times \prod_{w \not \in T \cup \{ \oo,v \} } G(\Qm_w)$ munies 
d'une action de $\HC_T$.

\begin{defi} \label{defi-gp-groth}
On dira que deux objets \og entiers \fg{} $A$ et $B$ 
de $\GF_{\bar \Qm_l}(v,h)$ ont mÍme rÈduction modulo $l$, et on Ècrira
$r_l(A)=r_l(B)$
si pour tout $T$ comme ci-avant,
les images de $A_T$ et $B_T$ dans $\GF_{\bar \Fm_l,T}(v,h)$ par rÈduction modulo $l$
sont Ègales. On dira que $A_T$ est la rÈduction modulo $l$ de niveau $T$ de $A$.
\end{defi}

\rem l'ensemble des rÈductions modulo $l$ de l'objet de $\GF_{\bar\Qm_l}(v,h)$ 
associÈes ‡ $\Pi^{\oo,v} \otimes \Pi'_v \otimes \chi$ 
n'est pas un objet de $\GF_{\bar \Fm_l}(v,h)=\lim_{\leftarrow T}
\GF_{\bar \Fm_l,T}(v,h)$ car la rÈduction modulo $l$ ne commute pas en gÈnÈral
avec le foncteur des invariants sous $G(\Zm_w)$.

\begin{lemm} \label{lem-foncteur}
Le foncteur $r_l:\groth_{\Qm_l}(GL_n(\Qm_p)) \longto \groth_{\Fm_l}(GL_n(\Qm_p))$
de rÈduction modulo $l$ du groupe de Grothendieck des $\Qm_l$-reprÈsentations entiËres de 
$GL_n(\Qm_p)$ dans celui des $\Fm_l$-reprÈsentations de $GL_n(\Qm_p)$ est tel que
$$r_l \Bigl ( \sum_i (-1)^i H^i(X,\Qm_l) \Bigr ) \simeq \sum_i (-1)^i H^i(X,\Fm_l).$$
\end{lemm}

\begin{proof} 
On choisit une rÈsolution Èquivariante plate 
$P^\bullet$ (i.e. sans $l$-torsion) de la $\Zm_l$ cohomologie de sorte que la cohomologie de 
$P^\bullet \otimes_{\Zm_l} \Qm_l$ (resp. $P^\bullet \otimes_{\Zm_l} \Fm_l$)
calcule les $H^\bullet(X,\Qm_l)$ (resp. $H^\bullet(X,\Fm_l)$). 
On rappelle alors le rÈsultat connu suivant (cf. par exemple \cite{boyer-ens} proposition 7.1.3):

\begin{lemm} Soit $(\DC^{\leq 0},\DC^{\geq 0})$ une catÈgorie dÈrivÈe munie d'une
$t$-structure non dÈgÈnÈrÈe: on note $\CC$ son coeur qui
est alors une catÈgorie abÈlienne de groupe de Grothendieck $\groth(\CC)$. L'application qui ‡ un objet $\FC$ de $\DC$ associe
$$\sum_i (-1)^i [\lexp p h^i \FC ] \in \groth(\CC)$$
induit un isomorphisme du groupe de Grothendieck $K(\DC)$ de la catÈgorie triangulÈe $\DC$, sur $\groth(\CC)$, o˘ l'on rappelle 
qu'Ètant donnÈe
une catÈgorie triangulÈe $A$ localement petite, son groupe de Grothendieck $K(A)$ est dÈfini
comme le groupe libre engendrÈ par les classes d'isomorphismes d'objets de $A$ quotientÈ, par les relations:
$A[1]=-A$ et $A=B+C$ pour tout triangle distinguÈ $B \longto A \longto C \longmapright{+1}$.
\end{lemm}

En considÈrant les $t$-structures triviales on en dÈduit alors que l'image de la somme alternÈe de la cohomologie de 
$P^\bullet \otimes_{\Zm_l} \Qm_l$ (resp. de  $P^\bullet \otimes_{\Zm_l} \Fm_l$) est Ègale dans le groupe de Grothendieck 
correspondant ‡ $\sum_i (-1)^i H^i(X,\Qm_l)$ (resp. $\sum_i (-1)^i H^i(X,\Fm_l)$). Ainsi le rÈsultat dÈcoule du fait Èvident que 
l'image par le morphisme de rÈduction modulo $l$ de la classe de $P^\bullet \otimes_{\Zm_l} \Qm_l$ dans $K(D^b(\spec \Fm_p,\Qm_l))$ est Ègale ‡ la 
classe de $P^\bullet \otimes_{\Zm_l} \Fm_l$ dans $K(D^b(\spec \Fm_p,\Fm_l))$.
\end{proof}

\subsection{Changement de cuspidalitÈ}

Rappelons les notations de \ref{nota-varrho}:
\begin{itemize}
\item $\varrho_{-1}$ est une $\bar \Fm_l$-reprÈsentation supercuspidale de $GL_{g_{-1}}(F_v)$;

\item pour $u \geq -1$ (resp. $u' \geq u$), on choisit $\pi_{v,u}$ (resp. $\pi_{v,u'}$) 
une $\bar \Qm_l$-reprÈsentation irrÈductible cuspidale de $GL_g(F_v)$ (resp. $GL_{g'}(F_v)$)
telle que, avec les notations de \ref{nota-varrho}, on ait $r_l(\pi_{v,u})=\varrho_u$ 
(resp. $r_l(\pi_{v,u'})=\varrho_{u'}$). En particulier on a $g=g_u=g_{-1} m(\varrho) l^u$
(resp. $g'=g_{u'}=g_{-1} m(\varrho)l^{u'}$). 
\end{itemize}

\rem dans la suite on Ècrira $l^{u'-u}$ alors que pour $u=-1$ on devrait Ècrire $m(\varrho) l^{u'}$; on espËre gagner
en clartÈ ce que l'on perdra en rigueur.

On choisit alors $r$ et $r'$ tels que $rg_u=r'g_{u'}$ de sorte que d'aprËs \ref{prop-red-D}, on a 
\addtocounter{smfthm}{1}
\begin{equation} \label{eq-chgt-cuspi}
l^{u'-u}\Bigl [ \Fm \FC_{\xi,\bar \Zm_l}(r,\pi_{v,u}) \Bigr ] =  \Bigl [ \Fm \FC_{\xi,\bar \Zm_l}(r',\pi_{v,u'}) \Bigr ].
\end{equation}
En revanche comme dans notre situation $\lexp p j_{!*}$ ne commute pas avec $\Fm$, cf. \cite{boyer-ens} \S 5,
nous allons utiliser le foncteur $j_!$ et donc l'ÈgalitÈ
\addtocounter{smfthm}{1}
\begin{multline} \label{eq-F-chgt2}
l^{u'-u} \Fm \Bigl [  j_!^{\geq rg} \FC_{\xi,\bar \Zm_l}(r,\pi_{v,u})\Bigr ] =l^{u'-u}
j_!^{\geq rg} \Bigl [ \Fm \FC_{\xi,\bar \Zm_l}(r,\pi_{v,u}) \Bigr ]  \\ = 
j_!^{\geq r'g'} \Bigl [ \Fm \FC_{\xi,\bar \Zm_l}(r',\pi_{v,u'}) \Bigr ] = \Fm \Bigl [ 
j_!^{\geq r'g'} \FC_{\xi,\bar \Zm_l}(r',\pi_{v,u'}) \Bigr ] .$$
\end{multline}

\begin{prop} \label{rl-hi}
Avec les notations prÈcÈdentes et 
au sens de la dÈfinition \ref{defi-gp-groth}, on a l'ÈgalitÈ
\begin{multline*}
l^{u'-u} \sum_i (-1)^i \sum_{\atop{(s,t)}{n_{s,t}(r,i)=1}} \sum_{\Pi \in \AC_{\xi,\pi_{v,u}}(s+t-1,s)} 
m(\Pi) d_\xi(\Pi_\oo) r_l \Bigl ( \Pi^{\oo,v} \otimes S_{\pi_{v,u}}(s,t)(r,i)(\Pi_v) \Bigr ) 
\\ = \\  
\sum_i (-1)^i \sum_{\atop{(s',t')}{n_{s',t'}(r',i)=1}} \sum_{\Pi' \in \AC_{\xi,\pi_{v,u'}}(s'+t'-1,s')} 
m(\Pi') d(\Pi'_\oo) r_l \Bigl ( \Pi^{',\oo,v} \otimes S_{\pi_{v,u'}}(s',t')(r',i)(\Pi'_v) \Bigr ).
\end{multline*}
\end{prop}

\begin{proof}
Le rÈsultat dÈcoule de la description explicite du corollaire \ref{coro-coho0} et du lemme \ref{lem-foncteur}.
\end{proof}

\rem en Ècrivant l'ÈgalitÈ prÈcÈdente qu'avec des signes positifs,
on voit qu'‡ un sous-quotient de niveau $T$ d'un $\Pi \in \AC_{\xi,\pi_{v,u}}(s+t-1,s)$ 
correspond un sous-quotient de niveau $T$ d'un $\Pi'$
\begin{enumerate}[(i)]
\item soit de $\AC_{\xi,\pi_{v,u'}}(s'+t'-1,s')$ pour un couple $(s',t')$ ‡ dÈterminer

\item soit de $\AC_{\xi,\pi_{v,u}}(s'+t'-1,s')$, et dans ce cas on n'est pas parvenu ‡ changer de cuspidale.
\end{enumerate}
Dans le cas (ii), si
$\Pi \in \AC_{\xi,\pi_{v,u}}(s+t-1,s)$ est associÈ ‡ un ÈlÈment de $\AC_{\xi,\pi_{v,u}}(s_{1}+t_{1}-1,s_1)$ 
avec $s_{1}+t_{1} > s+t$, en raisonnant par rÈcurrence on lui associe alors une reprÈsentation
$\AC_{\xi,\pi_{v,u'}}(s'+t'-1,s')$ comme dans le cas (i) ci dessus; autrement dit on parvient toujours 
‡ changer de cuspidale. Du cÙtÈ galoisien cela revient ‡ 
\og augmenter l'irrÈductibilitÈ locale\fg{} i.e. en notant $\sigma(\Pi)$ et $\sigma(\Pi')$
les reprÈsentations galoisiennes associÈes respectivement ‡ $\Pi$ et $\Pi'$ par la correspondance de Langlands globale, $\sigma(\Pi')$ est congruente ‡ $\sigma(\Pi)$
mais sa composante locale $\sigma(\Pi'_{v})$ est \og moins rÈductible \fg{} que
$\sigma(\Pi_{v})$ puisque certains de ses facteurs indÈcomposables sont devenus
irrÈductibles. Dans les paragraphes suivants nous
allons Ètudier les situations les plus simples de cette ÈgalitÈ.

\subsection{Le cas non dÈgÈnÈrÈ}

On reprend les notations du paragraphe prÈcÈdent.

\begin{prop} \label{prop-congru2}
Soit $\Pi' \in \AC_{\xi,\pi_{v,u'}}(t',1)$ une reprÈsentation irrÈductible automorphe. 
Il existe alors une 
reprÈsentation irrÈductible automorphe $\Pi \in \AC_{\xi,\pi_{v,u}}(t,1)$  
avec $tg=t'g'$, qui est congruente avec $\Pi'$.
\end{prop}

\begin{proof} 
Soit $T'$ l'ensemble des places finies $w$ de $F$ telles que $G(\Qm_w)$ est non ramifiÈ,
$G(\Zm_w)$ est un sous-groupe compact ouvert maximal de $G(\Qm_w)$ et 
$\Pi'_w$ est non ramifiÈe. Fixons $w_0 \in T'$ dÈcomposÈe dans $E$ et $v_0$ une
place de $F$ au dessus de $w_0$ telle que, avec la notation \ref{nota-GFv}, $G(F_{v_0})
\simeq GL_d(F_{v_0})$; on note $T=T'-\{ v_0 \}$. 
Notons $\lambda_T$ le caractËre de Satake de $\Pi'_T$
et soit $\overline{\lambda_T} \otimes \rho^{T,\oo,v} \otimes \varrho_v \otimes \chi$ 
un constituant irrÈductible de
niveau $T$ de $r_l\Bigl ( \Pi^{',\oo,v} \otimes S_{\pi_{v,u'}}(s',t')(r',i)(\Pi'_v) \Bigr )$ avec $\rho^{T,\oo,v}_{v_0}$ 
une reprÈsentation non dÈgÈnÈrÈe de $GL_d(F_{v_0})$. Pour toute reprÈsentation irrÈductible
cuspidale $\tilde \pi_v$ de $GL_{\tilde g}(F_v)$ et pour tout
$\Pi \in \AC_{\xi,\tilde \pi_v}(r,s)$ avec $s>1$ et $r$ quelconques,
$\rho^{T,\oo,v}_{v_0}$ n'est pas un constituant de $r_l(\Pi_{v_0})$, de sorte qu'en particulier pour tout
$i \neq 0$, la rÈduction modulo $l$ de $H^i(\FC_{\xi,\bar \Zm_l}(r, \tilde \pi_v)_1[d-r\tilde g])$ ne contient pas
 $\overline{\lambda_T} \otimes \rho^{T,\oo,v} \otimes \varrho_v \otimes \chi$.

Ainsi dans l'ÈgalitÈ de la proposition \ref{rl-hi} pour $r'=t'$,  $\overline{\lambda_T} \otimes \rho^{T,\oo,v} \otimes 
\varrho_v \otimes \chi$ apparait avec un signe strictement positif dans le membre de droite et donc dans
le membre de gauche de sorte qu'il existe $\Pi$
dont la rÈduction modulo $l$ de niveau $T$ est supÈrieure ou Ègale ‡ $\overline{\lambda_T} 
\otimes \rho^{T,\oo} \otimes \varrho_v \otimes \chi$.
\end{proof}

\subsection{Augmentation de l'irrÈductibilitÈ globale}
\label{para-supersingulier}

Dans le cas o˘ les faisceaux pervers dont la cohomologie donne l'ÈgalitÈ de \ref{rl-hi}, 
sont ‡ support dans les points supersinguliers, i.e. $rg_u=r'g_{u'}=d$, cette ÈgalitÈ s'Ècrit comme suit.

\begin{coro} 
Pour $(r,r')=(s,s')$ tels que $sg_u=s'g_{u'}=d$, \ref{rl-hi} s'Ècrit
\begin{multline*}
l^{u'-u} \sum_{\Pi \in \AC_{\xi,\pi_{v,u}}(s,1) \cup \AC_{\xi,\pi_{v,u}}(s,s)} m(\Pi)
d_\xi(\Pi_\oo) r_l(\Pi^{\oo,v}) \\ = 
\sum_{\Pi' \in \AC_{\xi,\pi_{v,u'}}(s',1)) \cup \AC_{\xi,\pi_{v,u'}}(s',s')} m(\Pi') 
d_\xi(\Pi'_\oo) r_l(\Pi^{',\oo,v}).
\end{multline*}
\end{coro}

\rem l'ÈgalitÈ est aussi valable dans le cas $d<(s+1)g$.

\begin{coro} \label{coro-congru2}
Supposons que dans le corollaire prÈcÈdent on ait $s'=1$ et
soit 
$\overline{\lambda_T} \otimes \rho^{T,\oo} \otimes \chi$ un sous-quotient irrÈductible de 
niveau $T$ de la rÈduction modulo $l$ de $\Pi$
o˘ $\Pi \in \AC_{\xi,\pi_{v,u}}(s,s)$. Il existe alors une reprÈsentation irrÈductible 
$\Pi' \in \AC_{\xi,\pi_{v,u'}}(1,1)$ dont 
la rÈduction modulo $l$ de niveau $T$ est supÈrieure ou Ègale ‡ 
$\overline{\lambda_T} \otimes \rho^{T,\oo,v} \otimes \chi$.
\end{coro}

\rem le thÈorËme 2 de \cite{sorensen} Ètablit des congruences entre une reprÈsentation $\Pi$ 
non ramifiÈe en $v$ telle son paramËtre de Satake est congruent ‡ celui de la triviale, 
et des reprÈsentations $\Pi'$ ramifiÈes en $v$, par exemple une Steinberg. Le corollaire
prÈcÈdent pour $g=1$ et $\pi_{v,u}=1_v$ construit une telle congruence sauf que 
nÈcessairement dans notre situation,
la reprÈsentation $\Pi$ est un caractËre. Notre technique ne permet pas passer d'un $\Pi_v \simeq \chi_1 \times \chi_2$ 
‡ une Steinberg; en revanche elle sait traiter d'autres cas que celui des paramËtres de
Satake congruent ‡ celui de la triviale.

\rem on peut voir le corollaire prÈcÈdent comme \emph{une augmentation de 
l'irrÈductibilitÈ globale} du cÙtÈ galoisien puisque, via la correspondance de Langlands
globale, ‡ $\Pi \in \AC_{\xi,\pi_{v,u}}(s,s)$ (resp. ‡ $\Pi' \in \AC_{\xi,\pi_{v,u'}}(1,1)$ tel qu'en 
une place $w$ quelconque $\Pi'_w$ est cuspidale)
correspond une reprÈsentation galoisienne
qui est la somme directe de $\frac{d}{s}$ reprÈsentations (resp. irrÈductible).
De ce point de vue l'ÈgalitÈ de la proposition \ref{rl-hi} peut s˚rement fournir d'autres 
rÈsultats intÈressants. Par exemple pour $r=s-1$ et $r'=s'$, elle s'Ècrit:
$$\begin{array}{l}
l^{u'-u} \sum_{\Pi \in \AC_{\xi,\pi_{v,u}}(s,1)} m(\Pi)d_\xi(\Pi_\oo) r_l \Bigl ( 
\Pi^{\oo,v} \otimes \tilde \pi_{v,u} \{ \frac{1-s}{2} \} \Bigr ) \\
+  l^{u'-u}  \sum_{\pi \in \AC_{\xi,\pi_{v,u}}(s-1,s-1)} m(\Pi)d_\xi(\Pi_\oo) r_l \Bigl ( 
\Pi^{\oo,v} \otimes  S_{\pi_{v,u}}(s-1,s-1)(s-1,0)(\Pi_v) \Bigr ) \\
+ l^{u'-u}   \sum_{\Pi \in \AC_{\xi,\pi_{v,u}}(s-1,1)} 
m(\Pi)d_\xi(\Pi_\oo) r_l \Bigl ( \Pi^{\oo,v} \otimes S_{\pi_{v,u}}(s-1,s-1)(s-1,0)(\Pi_v) \Bigr )
\\= \\ 
\sum_{\Pi' \in \AC_{\xi,\pi_{v,u'}}(s',1))}  m(\Pi')d_\xi(\Pi'_\oo) 
r_l \Bigl ( \Pi^{'\oo,v} \otimes S_{\pi_{v,u'}}(s',1) (s',0)(\Pi^{'\oo,v}) \Bigr ) \\
+  \sum_{\Pi' \in \AC_{\xi,\pi_{v,u'}}(s',s')} m(\Pi')d_\xi(\Pi'_\oo) 
r_l \Bigl ( \Pi^{',\oo,v} \otimes S_{\pi_{v,u'}}(s',1) (s',0)(\Pi^{'\oo,v}) \Bigr ) \\
+ l^{u'-u}  \sum_{\Pi \in \AC_{\xi,\pi_{v,u}}(s,s)} m(\Pi)d_\xi(\Pi_\oo) r_l \Bigl ( 
\Pi^{\oo,v} \otimes \tilde \pi_{v,u} \{ \frac{s-1}{2} \} \Bigr )
\end{array}$$
Ainsi les sous-quotients de niveau $T$ de la rÈduction modulo $l$ des
reprÈsentations de $\AC_{\xi,\pi_{v,u}}(s,s)$ doivent aussi apparaÓtre dans le membre de gauche
de l'ÈgalitÈ prÈcÈdente. En regardant en $v$, comme $\pi_{v,u}$ est cuspidale il ne peut
pas Ítre un sous-quotient de la rÈduction modulo $l$ d'une reprÈsentation de la forme 
$\speh_{s-1}(\pi')$. Ainsi on obtient des congruences automorphes entre les $\Pi \in \AC_{\xi,\pi_{v,u}}(s,s)$
et les $\Pi' \in \AC_{\xi,\pi_{v,u}}(s,1) \cup \AC_{\xi,\pi_{v,u}}(s-1,1)$ ce qui s'interprËte 
‡ nouveau en termes  \emph{d'augmentation de l'irrÈductibilitÈ globale} du 
cÙtÈ galoisien

\subsection{Torsion dans la cohomologie des faisceaux pervers d'Harris-Taylor}
\label{para-torsion}

Nous reprenons les notations prÈcÈdentes avec $u'>u= -1$.

\begin{prop} \label{prop-torsion} 
Pour tout $r' \geq 1$ tel que $r'g_{u'}\leq d-g_{-1}$,
la cohomologie de $j_{!}^{\geq r'g_{u'}} \FC_{\xi, \bar \Zm_l}(r',\pi_{v,u'})[d-r'g_{u'}]$ et de 
$j_{*}^{\geq r'g_{u'}} \FC_{\xi, \bar \Zm_l}(r',\pi_{v,u'})[d-r'g_{u'}]$
a de la torsion.
\end{prop}

\begin{proof} 
Rappelons que pour un $\Fm_q$-schÈma $X$ quelconque et un $\Zm_l$-faisceau pervers sans torsion $\PC$,
on a la suite exacte courte suivante
\addtocounter{smfthm}{1}
\begin{equation} \label{eq-ss-torsion}
0 \to H^n(X,\PC) \otimes_{\Zm_l} \Fm_l \longto H^n(X,\Fm \PC ) \longto H^{n+1}(X,\PC) [l] \to 0.
\end{equation}
On note $r$ tel que $rg_{-1}=r'g_{u'}$ et soit $s=\lfloor \frac{d}{g_{-1}} \rfloor$.
D'aprËs le cas $t=1$ de la proposition \ref{prop-coho-n0}, la partie libre de 
$H^{s-r}(j^{\geq rg_{-1}}_! \FC_{\xi,\bar \Zm_l}(r,\pi_{v,-1})[d-rg_{-1}])$ est non nulle.
Soit alors $i_0 \geq s-r$ l'entier $i$ maximal tel que 
$H^{i}(j^{\geq rg_{-1}}_! \FC_{\xi,\bar \Zm_l}(r,\pi_{v,-1})[d-rg_{-1}])$ est non nul; 
de la suite exacte (\ref{eq-ss-torsion}) on en
dÈduit que $i_0$ est aussi l'entier le plus grand tel que $H^i (j^{\geq rg_{-1}}_! \Fm 
\FC_{\xi,\bar \Zm_l}(r,\pi_{v,-1})[d-rg_{-1}])$ est non nul. 

\begin{itemize}
\item De (\ref{eq-chgt-cuspi}),
on dÈduit une filtration de $\Fm \bigl (j^{\geq r'g_{u'}}_! \FC_{\xi,\bar \Zm_l}(r',\pi_{v,u'}) \bigr )$
par des $j^{\geq rg_{-1}}_! \Fm \FC_{\xi,\bar \Zm_l}(r,\pi_{v,-1})$. 

\item De la suite spectrale associÈe qui calcule les 
$$H^i \Bigl ( \Fm \bigl (j^{\geq r'g_{u'}}_! \FC_{\xi,\bar \Zm_l}(r',\pi_{v,u'})[d-r'g_{u'}] \bigr ) \Bigr ),$$ 
on en dÈduit que $i_0$ est aussi le plus grand entier tel que ce groupe de cohomologie est non nul.
\end{itemize}
D'aprËs \ref{prop-coho-n0}, pour $s'=\lfloor \frac{d}{g_{u'}} \rfloor$,
le quotient libre de $H^i \Bigl ( j^{\geq r'g_{u'}}_! \FC_{\xi,\bar \Zm_l}(r',\pi_{v,u'})[d-r'g_{u'}] \Bigr )$ est nul, dËs que $i>s'-r'$.
Or comme par hypothËse $rg_{-1}+g_{-1} \leq d$, on a $i_0 \geq s-r> s'-r'$ et donc d'aprËs (\ref{eq-ss-torsion})
la torsion de $H^{i_0} ( j^{\geq r'g_{u'}}_! \FC_{\xi,\bar \Zm_l}(r',\pi_{v,u'})[d-r'g_{u'}])$ est non nulle. Par dualitÈ on en dÈduit que $H^{-i_0+1}( j^{\geq r'g_{u'}}_* \FC_{\xi,\bar \Zm_l}(r',\pi_{v,u'})[d-r'g_{u'}])$ a aussi de la torsion.
\end{proof}

\rem 
avec les notations de la preuve prÈcÈdente, si $i_0>s-r$ alors 
$j^{\geq rg_{-1}}_! \FC_{\xi,\bar \Zm_l}(r,\pi_{v,-1})[d-rg_{-1}]$ a aussi 
de la torsion dans sa cohomologie.

\section{ReprÈsentations automorphes fortement congruentes}
\label{para-rep-congru-forte}

\subsection{DÈfinition de la notion de congruence au sens fort.}

\begin{defi} \label{defi-congru-forte}
Deux $\bar \Qm_l$-reprÈsentations irrÈductibles entiËres $\Pi$ et $\Pi'$ de $GL_d(\Am_F)$ 
sont dites \emph{fortement congruentes modulo $l$} si pour toute place finie $w$ de $F$ telle 
que $\Pi_w$ et $\Pi'_w$ sont non ramifiÈes, les rÈductions modulo $l$ de $\Pi_w$ et $\Pi'_w$ 
sont Ègales.
\end{defi}

\rem si pour une place $w$ o˘ $\Pi$ et $\Pi'$ sont non ramifiÈes, $r_l(\Pi_w)$ et 
$r_l(\Pi'_w)$ sont Ègales alors la rÈduction modulo 
$l$ de leurs caractËres de Satake aussi, de sorte que deux reprÈsentations fortement
congruentes sont aussi faiblement congruentes au sens du \S \ref{para-formulation}. 
En ce qui concerne la rÈciproque remarquons que:
\begin{itemize}
\item si $\Pi_w$ et $\Pi'_w$ sont irrÈductibles non ramifiÈes et non dÈgÈnÈrÈes 
alors $r_l(\Pi_w)= r_l(\Pi'_w)$ si et seulement si leurs caractËres de Satake sont congrus 
modulo $l$. 

\item Plus gÈnÈralement si $\Pi_w$ et $\Pi'_w$ sont irrÈductibles non ramifiÈes,
respectivement de la forme $\speh_s(\pi_w)$ et $\speh_s(\pi'_w)$ avec $\pi_w$ et 
$\pi'_w$ non dÈgÈnÈrÈes, alors $r_l(\Pi_w)=r_l(\Pi'_w)$ si et seulement si leurs caractËres
de Satake sont congrus modulo $l$.

\item Par contre $\Pi_w\simeq \chi_1 \times \cdots \times \chi_d$ irrÈductible non ramifiÈe 
et non dÈgÈnÈrÈe peut trËs bien avoir un caractËre de Satake congru modulo $l$ ‡ celui 
de la triviale.
\end{itemize}

\begin{defi} \label{defi-forte-congru}
Deux $\bar \Qm_l$-reprÈsentations irrÈductibles entiËres $\Pi$ et $\Pi'$ de $G(\Am)$ 
sont dites \emph{fortement congruentes modulo $l$} si les reprÈsentations
$\Pi_{GL}$ et $\Pi'_{GL}$ de $GL_d(\Am_F)$ obtenues, ‡ partir de respectivement
$\Pi$ et $\Pi'$, par la correspondance de Jacquet-Langlands de leur changement de base,
cf. le \S VI de \cite{h-t}, sont fortement congruentes au sens de la dÈfinition prÈcÈdente.
\end{defi}

\rem il rÈsulte des remarques prÈcÈdentes, que si $\Pi$ et $\Pi'$ sont fortement congruentes
alors elles sont faiblement congruentes au sens du \S \ref{para-formulation}.

\begin{prop} \label{prop-forte-congru}
Soient $\Pi$ et $\Pi'$ des $\bar \Qm_l$-reprÈsentations irrÈductibles automorphes 
de $G(\Am)$, entiËres et $\xi$-cohomologiques. On note $T$ l'ensemble des places
finies de $\Qm$ o˘ $\Pi$ et $\Pi'$ sont non ramifiÈes et on suppose que 
\begin{itemize}
\item les caractËres de Hecke $\lambda_{T,\Pi}$ et $\lambda_{T,\Pi'}$ sont congrus 
modulo $l$;

\item qu'il existe une place $w_0$ de $T$ dÈcomposÈe 
$w_0=u_0u_0^c$ dans $E$ ainsi qu'une place $v_0$ de $F$ au dessus de $u_0$
telle qu'avec la notation \ref{nota-GFv}, $G(F_{v_0}) \simeq GL_d(F_{v_0})$ avec
$$\Pi_{w_0} \simeq \speh_s(\pi_{w_0}) \hbox{ et } \Pi'_{w_0} \simeq \speh_s(\pi'_{w_0})$$
o˘ $\pi_{w_0}$ et $\pi'_{w_0}$ sont des reprÈsentations irrÈductibles non dÈgÈnÈrÈes.
\end{itemize}
Alors $\Pi$ et $\Pi'$ sont fortement congruentes.
\end{prop}

\begin{proof}
Il dÈcoule du deuxiËme tiret de la remarque prÈcÈdant la dÈfinition \ref{defi-forte-congru},
qu'en toute place $w$ de $F$ pour laquelle $\Pi_{GL}$ et $\Pi'_{GL}$ sont non ramifiÈes,
les rÈductions modulo $l$ de $\Pi_{GL,w}$ et $\Pi'_{GL,w}$ sont Ègales. 
\end{proof}

\subsection{\'EnoncÈ de la conjecture automorphe}
\label{para-congruence1}

DÈsormais $\pi_v$ et $\pi'_v$ seront deux reprÈsentations irrÈductibles cuspidales de 
$GL_g(F_v)$ ayant la mÍme rÈduction modulo $l$.
On se propose de montrer l'existence de couples $(\Pi,\Pi') \in \AC_{\xi,\pi_v}(r,s) \times \AC_{\xi,\pi'_v}(r,s)$ 
de reprÈsentations fortement congruentes o˘ $r_l(\pi_v)=r_l(\pi'_v)$, 
cf. la notation \ref{nota-AC},.
Afin d'allÈger les notations, pour $\Pi \in \AC_{\xi,\pi_v}(r,s)$ et pour tout $(r',i)$ tels que $m(r',i)=1$, on notera 
$t_{r',i}^{r,s} (\Pi)=\Pi^{\oo,v} \otimes \Bigl ( S_{\pi_v}(s,r-s+1)(r',i) \bigl ( \Pi_v \bigr )\Bigr ).$

\begin{conj} \label{conj}
Pour tout $r \geq s$, au sens de \ref{defi-gp-groth}, on a
$$\sum_{\Pi \in \AC_{\xi,\pi_v}(r,s)}m(\Pi) d_\xi(\Pi_\oo) r_l \Bigl ( t^{r,s}_{r,0} (\Pi) \Bigr ) =
\sum_{\Pi' \in \AC_{\xi,\pi'_v}(r,s)}m(\Pi') d_\xi(\Pi'_\oo) r_l \Bigl ( t^{r,s}_{r,0} (\Pi') \Bigr ).$$
\end{conj}

\rem pour $\overline{\lambda_T} \otimes \varrho^{T,\oo,v} \otimes \varrho_v \otimes \chi$ 
un irrÈductible de $G_{\bar \Fm_l,T}(v,h)$, on 
choisit pour tout $z \not \in T \cup \{ \oo \}$ un sous-groupe compact ouvert $K_z$ de 
$G(\Qm_z)$ tel que $\rho^{T,\oo}_z$ a des vecteurs invariants sous $K_z$.
Toute reprÈsentation $\Pi$ de $G(\Am^{\oo,v}) \times GL_h(F_v)$ dont la rÈduction
modulo $l$ de niveau $T$ contient 
$\overline{\lambda_T} \otimes \varrho^{T,\oo} \otimes \varrho_v \otimes \chi$,
admet, par dÈfinition, des vecteurs invariants sous $K^\oo:= G(\Qm_T) \times
\prod_{z \not \in T \cup \{ \oo \} } K_z$. En particulier il n'y a qu'un nombre fini
de reprÈsentations $\Pi \in \AC_{\xi,\pi_v}(r,s)$ telles que $r_l(t_{r,0}^{r,s}(\Pi))$ contiennent 
$\overline{\lambda_T} \otimes \varrho^{T,\oo} \otimes \varrho_v \otimes \chi$.

\begin{defi} \label{defi-n}
La multiplicitÈ de $\overline{\lambda_T} \otimes \varrho^{T,\oo} \otimes \varrho_v \otimes 
(\chi \otimes \Xi^{\frac{2i+r'-r}{2}})$ dans 
$\sum_{\Pi \in \AC_{\xi,\pi_v}(r,s)} m(\Pi) d_\xi(\Pi_\oo) r_l (t^{r,s}_{r',i}(\Pi))$
est finie, on la notera
$n^{(r,s)}_{(r',i),\pi_v}(\overline{\lambda_T} \otimes \varrho^{T,\oo} \otimes \varrho_v 
\otimes \chi)$.
\end{defi}

\rem le lecteur notera que pour $\varrho_v \otimes \chi$ et $\varrho'_v \otimes \chi'$ 
deux constituants irrÈductibles
de $r_l(S_{\pi_v}(s,r-s+1)(r',i))$ de multiplicitÈ respectives $n$ et $n'$, on a l'ÈgalitÈ
$$n'. n^{(r,s)}_{(r',i),\pi_v}(\overline{\lambda_T} \otimes \varrho^{T,\oo} \otimes \varrho_v
\otimes \chi)= n .n^{(r,s)}_{(r',i),\pi_v}(\overline{\lambda_T} \otimes \varrho^{T,\oo} 
\otimes \varrho'_v \otimes \chi').$$
La conjecture \ref{conj} s'Ècrit alors comme une famille d'ÈgalitÈs
\addtocounter{smfthm}{1}
\begin{equation} \label{eq-rs}
n^{(r,s)}_{(r,0),\pi_v}(\overline{\lambda_T} \otimes \varrho^{T,\oo} \otimes \varrho_v 
\otimes \chi) =
n^{(r,s)}_{(r,0),\pi'_v}(\overline{\lambda_T} \otimes \varrho^{T,\oo} \otimes \varrho_v
\otimes \chi),
\end{equation}
indexÈe par les $\overline{\lambda_T} \otimes \varrho^{T,\oo} \otimes \varrho_v \otimes \chi$. 
Ces ÈgalitÈs sont ‡ interprÈter en termes de congruences fortes entre reprÈsentations
automorphes. PrÈcisÈment pour 
$\overline{\lambda_T} \otimes \varrho^{T,\oo} \otimes \varrho_v \otimes \chi$ un irrÈductible 
de $\GF_{\bar \Fm_l,T}(v,d-rg)$ tel que 
$n^{(r,s)}_{(r,0),\pi_v}(\overline{\lambda_T} \otimes \varrho^{T,\oo} \otimes \varrho_v 
\otimes \chi) \neq 0$, 
il existe $\Pi \in \AC_{\xi,\pi_v}(r,s)$, et donc $\Pi' \in \AC_{\xi,\pi'_v}(r,s)$, 
dont les rÈductions modulo $l$ de niveau $T$ de $t^{r,s}_{r,0}(\Pi)$ et $t^{r,s}_{r,0}(\Pi')$
contiennent $\overline{\lambda_T} \otimes \varrho^{T,\oo}  \otimes \varrho_v \otimes \chi$. 
Alors d'aprËs
la proposition \ref{prop-forte-congru}, $\Pi$ et $\Pi'$ sont \emph{fortement congruentes}.
En outre pour toute place $w$, les rÈductions modulo $l$ de $\Pi_w$ et $\Pi'_w$ ne 
sont pas disjointes: c'est clair pour toutes les places $p' \neq p$ de $\Zm$
puisqu'elles contiennent $\rho^{T,\oo}_{p'}$. Pour la place $p$, c'est ‡ nouveau clair
pour les places $v_i \neq v$ et reste alors ‡ traiter le cas de la place $v$.
\'Ecrivons
$\Pi_v \simeq \speh_s(\st_t(\pi_{1,v})) \times \cdots \times \speh_s(\st_t(\pi_{k,v})) \times ?$
o˘ les $\pi_{i,v}\simeq \pi_v \otimes \xi_i \circ \val \circ \det$ pour $i=1,\cdots,k$ sont 
inertiellement Èquivalents ‡ $\pi_v$ et 
\begin{multline*}
t_{r,0}^{r,s}(\Pi)=\Pi^{\oo,v} \otimes \sum_{i=1}^k \Bigl ( \speh_s(\st_t(\pi_{1,v})) \times \cdots
\times  \speh_s(\st_t(\pi_{i-1,v})) \\
\times S_{\pi_{i,v}}(s,r-s+1)(r,0) \times \\
\speh_s(\st_t(\pi_{i+1,v})) \times \cdots \times  \speh_s(\st_t(\pi_{k,v})) \times ? \Bigr ) 
\otimes \xi_i.
\end{multline*}
Notons $i$ un indice pour lequel $\varrho_v \otimes \chi$ est un constituant de la 
rÈduction modulo $l$ de 
$A_v:=\Bigl ( \speh_s(\st_t(\pi_{1,v})) \times \cdots \times S_{\pi_{i,v}}(s,r-s+1)(r,0) 
\times \cdots \times  \speh_s(\st_t(\pi_{k,v})) \times ? \Bigr ) \otimes \xi_i$. Du cÙtÈ de $\Pi'$,
$\varrho_v \otimes \chi$ est un constituant de la rÈduction modulo $l$ de
$A'_v:=\Bigl ( \speh_s(\st_t(\pi'_{1,v})) \times \cdots \times S_{\pi'_{i',v}}(s,r-s+1)(r,0) 
\times \cdots \times  \speh_s(\st_t(\pi'_{k',v})) \times ? \Bigr ) \otimes \xi'_{i'}$. Ainsi
la rÈduction modulo $l$ de $\xi_i$ et $\xi'_{i'}$ sont Ègales et donc, comme $r_l(\pi_v)=
r_l(\pi'_v)$, les rÈductions modulo $l$ de $\pi_{i,v}$ et $\pi'_{i',v}$ sont Ègales. On en dÈduit
que les rÈductions modulo $l$ de $S_{\pi_{i,v}}(s,r-s+1)(r,0)$ et $S_{\pi'_{i',v}}(s,r-s+1)(r,0)$
sont Ègales et comme celles de $\speh_s(\st_t(\pi_{i,v}))$ et $\speh_s(\st_t(\pi'_{i',v}))$
sont aussi Ègales, du fait que $r_l(A_v)$ et $r_l(A'_v)$ ont un constituant en commun, on
en dÈduit qu'il en est de mÍme pour les rÈductions modulo $l$ de $\Pi_v$ et $\Pi'_v$.

\subsection{Cas trËs ramifiÈs}
\label{para-ramifie}

Nous nous proposons d'Ètudier dans ce paragraphe le cas trËs simple o˘
$\overline{\lambda_T} \otimes \varrho^{T,\oo}$ vÈrifie la propriÈtÈ suivante:
\addtocounter{smfthm}{1}
\begin{equation} \label{eq-propriete0}
\forall \pi_v, ~ \forall \varrho_v \otimes \chi, \quad n^{(r,s)}_{(r',i),\pi_v}(\overline{\lambda_T} 
\otimes \varrho^{T,\oo} \otimes \varrho_v \otimes \chi) \neq 0 \Leftrightarrow s=1 \hbox{ et } 
i=0.
\end{equation}
Citons comme cas particuliers:
\begin{itemize}
\item celui o˘ la reprÈsentation $\xi$ est trËs ramifiÈe de sorte que, la cohomologie
de la variÈtÈ de Shimura Ètant concentrÈe en degrÈ mÈdian, d'aprËs \cite{boyer-compositio}
\S 4, pour tout $s >1$ les $\AC_{\xi,\pi_v}(r,s)$ sont vides.

\item celui o˘ il existe une place $x \not \in T \cup \{ \oo \}$ 
dÈcomposÈe $x=yy^c$ dans $E$ et une place
$z$ de $F$ au dessus de $y$ telle que, au sens de la notation \ref{nota-GFv}, 
$G(F_z) \simeq GL_d(F_z)$ et $\rho^{T,\oo}_z$ est non dÈgÈnÈrÈe.
\end{itemize}

Pour un $\overline{\lambda_T} \otimes \varrho^{T,\oo}$ vÈrifiant la propriÈtÈ 
(\ref{eq-propriete0}), nous allons prouver les ÈgalitÈs (\ref{eq-rs}) par rÈcurrence sur $r$ de
$s_g:=\lfloor \frac{d}{g} \rfloor$ ‡ $1$. Pour $r=s_g$, les couples $(s,t)$
tel qu'il existe $i$ pour lequel $m_{s,t}(s_g,i) \neq 0$ sont $(s_g,1)$ et $(1,s_g)$
ce qui correspond aux reprÈsentations 
\begin{itemize}
\item $\Pi \in \AC_{\xi,\pi_v}(s_g,s_g)$ et donc $\Pi_v \simeq \speh_{s_g}(\pi_v) \times ?$,  

\item $\Pi \in \AC_{\xi,\pi_v}(s_g,1)$ et donc $\Pi_v \simeq \st_{s_g}(\pi_v) \times ?$,
\end{itemize}
avec $?$ une reprÈsentation de $GL_k(F_v)$ avec $k < g$. Rappelons que les extensions
intermÈdiaires et par zÈro de $\FC_\xi(s_g,\pi_v)[d-s_g g]$ sont Ègales et qu'alors
la $\bar \Qm_l$-cohomologie est concentrÈe en degrÈ mÈdian. D'aprËs le corollaire
\ref{coro-coho0}, $\frac{d}{e_{\pi_v} \sharp \ker^1(\Qm,G)}
[H^0( j_{!}^{\geq s_g g} \FC_\xi(s_g,\pi_v)_1[d-s_g g])]$ est Ègal ‡
\addtocounter{smfthm}{1}
\begin{multline} \label{eq-cas-ss}
\sum_{\Pi \in \AC_{\xi,\pi_v}
(s_g,s_g)} m(\Pi)d_\xi(\Pi_\oo)  \Pi^{\oo,v} \otimes S_{\pi_v}(s_g,s_g)(s_g,0)(\Pi_v) \\ 
+ \sum_{\Pi \in \AC_{\xi,\pi_v}(s_g,1)} m(\Pi)d_\xi(\Pi_\oo) \Pi^{\oo,v} \otimes
S_{\pi_v}(s_g,1)(s_g,0)(\Pi_v).
\end{multline}
Le cas de $n^{(s_g,1)}_{(s_g,0),\pi_v}(\overline{\lambda_T} \otimes \varrho^{T,\oo}\otimes 
\varrho_v \otimes \chi)$
dÈcoule alors de la proposition suivante en utilisant la nullitÈ de
$n^{(s_g,s_g)}_{(s_g,0),\pi_v}(\overline{\lambda_T} \otimes \varrho^{T,\oo}\otimes \varrho_v
\otimes \chi)$ et
$n^{(s_g,s_g)}_{(s_g,0),\pi'_v}(\overline{\lambda_T} \otimes \varrho^{T,\oo}\otimes \varrho_v
\otimes \chi)$.

\begin{prop} \label{prop-ss-2}
Au sens de la dÈfinition \ref{defi-gp-groth}, on a
\begin{multline*}
\sum_{\Pi \in \AC_{\xi,\pi_v}(s_g,s_g) \bigcup \AC_{\xi,\pi_v}(s_g,1)} m(\Pi)
d_\xi(\Pi_\oo) r_l ( \Pi^\oo  ) \\ = 
\sum_{\Pi' \in \AC_{\xi,\pi'_v}(s_g,s_g) \bigcup \AC_{\xi,\pi'_v}(s_g,1)} 
m(\Pi') d_\xi(\Pi'_\oo) r_l  (\Pi^{',\oo}).
\end{multline*}
\end{prop}

\begin{proof} Notons que l'on a
$\Fm \FC_\xi (s_g,\pi_v) = \Fm \FC_\xi(s_g,\pi'_v)$
et que le foncteur $\Fm$ commute avec $j_!^{\geq s_g g}$. On en dÈduit alors que,
pour tout sous-groupe compact ouvert $I=G(\Zm_T) \times I^{T,\oo}$ de $G(\Am^\oo)$, on a
$$\Fm H^{*} ( \overline X_I, j_{!}^{\geq s_g g} \FC_\xi(s_g,\pi_v)_1[d-s_gg]) =
H^{*} \Bigl ( \overline X_I, \Fm j_{!}^{\geq s_g g} \FC_\xi(s_g,\pi_v)_1[d-s_gg] \Bigr ).$$
D'aprËs (\ref{eq-cas-ss}), on a la formule pour $\Pi^{\oo,v} \otimes 
S_{\pi_v}(s_g,s_g)(s_g,0)(\Pi_v)$ (resp.  $\Pi^{\oo,v} \otimes S_{\pi_v}(s_g,1)(s_g,0)(\Pi_v)$)
‡ la place de $\Pi^\oo$ pour $\Pi \in \AC_{\xi,\pi_v}(s_g,s_g)$ (resp. 
pour $\Pi \in \AC_{\xi,\pi_v}(s_g,1)$). Pour passer de $S_{\pi_v}(s_g,s_g)(s_g,0)(\Pi_v)$
(resp. de $S_{\pi_v}(s_g,1)(s_g,0)(\Pi_v)$) ‡ $\Pi_v$, on raisonne comme prÈcÈdemment
en utilisant que la rÈduction modulo $l$ de
$\speh_{s_g}(\pi_v)$ (resp. $\st_{s_g}(\pi_v)$) ne dÈpend que de celle de $\pi_v$ et de $s_g$.
\end{proof}

\begin{lemm}
Soit $\varrho_v \otimes \chi$ un constituant irrÈductible de 
$S_{\pi_v}(s,r-s+1)(r,0) \bigl ( \Pi_v \bigr )$,si
$$ n^{(r,s)}_{(r,0),\pi_v}(\overline{\lambda_T} \otimes \varrho^{T,\oo} \otimes \varrho_v
\otimes \chi) =
n^{(r,s)}_{(r,0),\pi'_v}(\overline{\lambda_T} \otimes \varrho^{T,\oo} \otimes \varrho_v
\otimes \chi)$$
alors pour tout $(r',i)$ et pour toute reprÈsentation $\varrho'_v \otimes \chi'$ de 
$GL_{d-r'g}(F_v) \times \Zm$, on a
$$ n^{(r,s)}_{(r',i),\pi_v}(\overline{\lambda_T} \otimes \varrho^{T,\oo} \otimes \varrho'_v
\otimes \chi')=
n^{(r,s)}_{(r',i),\pi_v}(\overline{\lambda_T} \otimes \varrho^{T,\oo} \otimes \varrho'_v
\otimes \chi').$$
\end{lemm}

\begin{proof} 
Si $m_{r,s}(r',i)=0$ c'est bien entendu Èvident. Sinon le rÈsultat dÈcoule du fait que
la rÈduction modulo $l$ de $S_{\pi_v}(s,t)(r',i)$ est Ègale ‡ celle de $S_{\pi'_v}(s,t)(r',i)$.
\end{proof}

On suppose ‡ prÈsent, par rÈcurrence, que pour tout $r'>r$ et pour tout $s$, on a
$$n^{(r',s)}_{(r',0),\pi_v}(\overline{\lambda_T} \otimes \varrho^{T,\oo} \otimes \varrho_v
\otimes \chi)=
n^{(r',s)}_{(r',0),\pi'_v}(\overline{\lambda_T} \otimes \varrho^{T,\oo} \otimes \varrho_v
\otimes \chi).$$
Du lemme suivant et de la nullitÈ, d'aprËs la propriÈtÈ (\ref{eq-propriete0}), des 
$ n^{(r,s)}_{(r,0),\pi_v}(\overline{\lambda_T} \otimes \varrho^{T,\oo} \otimes \varrho_v
\otimes \chi)$ pour tout $s \geq 2$, on obtient l'ÈgalitÈ ci-dessus pour $r$ ce qui termine
la preuve de l'induction de la rÈcurrence.

\begin{lemm} \label{lem-init}
L'ÈgalitÈ suivante est vÈrifiÈe
\addtocounter{smfthm}{1}
\begin{equation} \label{eq-a-distinguer}
\sum_s  n^{(r,s)}_{(r,0),\pi_v}(\overline{\lambda_T} \otimes \varrho^{T,\oo} \otimes \varrho_v
\otimes \chi)=
\sum_s  n^{(r,s)}_{(r,0),\pi'_v}(\overline{\lambda_T} \otimes \varrho^{T,\oo} \otimes \varrho_v
\otimes \chi).
\end{equation}
\end{lemm}

\begin{proof} 
D'aprËs la proposition \ref{prop-coho-gen}, le corollaire prÈcÈdent et le corollaire 
\ref{coro-di}, pour tout $i \neq 0$, les multiplicitÈs de 
$\overline{\lambda_T} \otimes \varrho^{T,\oo} \otimes \varrho_v \otimes \chi$ dans
la rÈduction
modulo $l$ des parties libres de $H^i(\lexp p j_{!}^{\geq rg} \FC_{\xi}(\pi_v,r)[d-rg])$ et 
$H^i(\lexp p j_{!}^{\geq rg} \FC_{\xi}(\pi'_v,r)[d-rg])$ sont Ègales. Par ailleurs d'aprËs le lemme 
\ref{lem-foncteur}, la somme alternÈe de la rÈduction modulo $l$ des groupes de
cohomologie ne dÈpend que des parties libres, de sorte que l'ÈgalitÈ
$$\sum_i (-1)^i  r_l \Bigl (H^{i} ( \lexp p j^{\geq rg}_{!} \FC_{\xi}(\pi_v,r)_1[d-rg]) \Bigr )= 
\sum_i (-1)^i r_l \Bigl ( H^{i} ( \lexp p j^{\geq rg}_{!} \FC_{\xi}(\pi'_v,r)_1[d-rg]) \Bigr )$$
et ce qui prÈcËde, nous donne que les multiplicitÈs de
$\overline{\lambda_T} \otimes \varrho^{T,\oo} \otimes \varrho_v \otimes \chi$
dans la rÈduction modulo $l$
des parties libres de $H^{0} ( \lexp p j^{\geq rg}_{!} \FC_{\xi}(\pi_v,r)_1[d-rg])$ et de $H^{0} 
( \lexp p j^{\geq tg}_{!} \FC_{\xi}(\pi'_v,r)_1[d-rg])$ sont Ègales, ce qui d'aprËs 
\ref{prop-coho-gen} donne le rÈsultat annoncÈ.
\end{proof}

\subsection{Une conjecture sur la torsion}

Comme on peut le voir dans le paragraphe prÈcÈdent, la propriÈtÈ (\ref{eq-propriete0}),
nous sert ‡ sÈparer les contributions des diffÈrents $s$ dans (\ref{eq-a-distinguer}).
Donnons une motivation pour croire ‡ \ref{conj} en Ètudiant le cas $r=s_g$ de la 
proposition \ref{prop-ss-2}. Supposons alors par l'absurde que
\addtocounter{smfthm}{1}
\begin{equation} \label{eq-absurde}
n^{(s_g,1)}_{(s_g,0),\pi_v}(\overline{\lambda_T} \otimes \varrho^{T,\oo}\otimes \varrho_v
\otimes \chi) >
n^{(s_g,1)}_{(s_g,0),\pi'_v}(\overline{\lambda_T} \otimes \varrho^{T,\oo}\otimes \varrho_v
\otimes \chi).
\end{equation}

\begin{lemm} \label{lem-congru-forte}
Soit $T_1$ l'ensemble infini des places $w$ de $F$ au dessus d'une place $x$ 
de $T$ dÈcomposÈe $x=yy^c$ dans $E$, telles qu'au sens de la notation \ref{nota-GFv},
$G(F_w) \simeq GL_d(F_w)$. Alors pour tout sous-ensemble fini 
$S_1 \subset T_1$ de cardinal pair (resp. impair), il existe une reprÈsentation 
$\Pi_{S_1} \in  \AC_{\xi,\pi_v}(s_g,1)$ (resp. $\Pi'_{S_1} \in  \AC_{\xi,\pi'_v}(s_g,1)$) ramifiÈe 
en toutes les places de $T_1$, de sorte que toutes les reprÈsentations $\Pi_{S_1}$ et
$\Pi'_{S_1}$ ainsi construites sont fortement congruentes entre elles.
\end{lemm}

\begin{proof}
Pour $w \in T_1$ et $\Pi_w$ une reprÈsentation de $GL_d(F_w)$ non ramifiÈe
et non dÈgÈnÈrÈe dont la rÈduction modulo $l$ du caractËre de Satake est 
$\overline{\lambda_w}$, $r_l(\Pi_w)$ qui ne dÈpend que de $\overline{\lambda_w}$,
contient une unique reprÈsentation non dÈgÈnÈrÈe que l'on note $\pi_w$. 
On regarde la multiplicitÈ de $\overline{\lambda_{T^w}} \otimes \varrho^{T,\oo}\otimes 
\pi_w \otimes \varrho_v \otimes \chi$ dans l'ÈgalitÈ de \ref{prop-ss-2} 
au niveau $T^w=T-\{ w \}$ en distinguant les contributions des $\Pi \in \AC_{\xi,\pi_v}(s_g,1)$
qui sont ou non ramifiÈes en $w$, en notant que pour tout $\Pi \in \AC_{\xi,\pi_v}(s_g,s_g)$,
$\pi_w$ n'est pas un constituant de $r_l(\Pi_w)$:
\begin{multline*}
n^{(s_g,1)}_{(s_g,0),\pi_v}(\overline{\lambda_T} \otimes \varrho^{T,\oo}\otimes \varrho_v
\otimes \chi) + m_w(\overline{\lambda_T} \otimes \varrho^{T,\oo}\otimes \varrho_v
\otimes \chi) \\ =
n^{(s_g,1)}_{(s_g,0),\pi'_v}(\overline{\lambda_T} \otimes \varrho^{T,\oo}\otimes \varrho_v
\otimes \chi) + m'_w(\overline{\lambda_T} \otimes \varrho^{T,\oo}\otimes \varrho_v
\otimes \chi)
\end{multline*}
o˘ $m_w(\overline{\lambda_T} \otimes \varrho^{T,\oo}\otimes \varrho_v \otimes \chi)$
correspond donc ‡ la multiplicitÈ de $\overline{\lambda_{T^w}} \otimes \varrho^{T,\oo}\otimes 
\pi_w \otimes \varrho_v \otimes \chi$ dans la rÈduction modulo $l$ de niveau $T^w$
des $\Pi \in \AC_{\xi,\pi_v}(s_g,1)$ qui sont ramifiÈes en $w$ et non ramifiÈes en $T^w$.
De l'inÈgalitÈ (\ref{eq-absurde}), on en dÈduit que 
$m'_w(\overline{\lambda_T} \otimes \varrho^{T,\oo}\otimes \varrho_v \otimes \chi) >0$
ce qui prouve le cas o˘ $S_1$ est de cardinal $1$.

Pour tout $S_1$, on note 
$m_{S_1}(\overline{\lambda_T} \otimes \varrho^{T,\oo}\otimes \varrho_v \otimes \chi)$
la multiplicitÈ de $\overline{\lambda_{T-S_1}} \otimes \varrho^{T,\oo}\otimes 
(\prod_{w \in S_1} \pi_w) \otimes \varrho_v \otimes \chi$ dans la rÈduction modulo $l$ 
de niveau $T-S_1$ des $\Pi \in \AC_{\xi,\pi_v}(s_g,1)$ qui sont ramifiÈes en $S_1$ et
non ramifiÈes en $T-S_1$. On raisonne alors par rÈcurrence sur le cardinal de $S_1$
en supposant que pour tout $S'_1$ de cardinal $\leq r$
\begin{multline*}
n^{(s_g,1)}_{(s_g,0),\pi_v}(\overline{\lambda_T} \otimes \varrho^{T,\oo}\otimes \varrho_v
\otimes \chi) + (-1)^{\sharp S'_1} m_{S'_1}(\overline{\lambda_T} \otimes \varrho^{T,\oo}
\otimes \varrho_v \otimes \chi) \\ =
n^{(s_g,1)}_{(s_g,0),\pi'_v}(\overline{\lambda_T} \otimes \varrho^{T,\oo}\otimes \varrho_v
\otimes \chi) + (-1)^{\sharp S'_1} m'_{S'_1}(\overline{\lambda_T} \otimes \varrho^{T,\oo}
\otimes \varrho_v \otimes \chi)
\end{multline*}
L'ÈgalitÈ de la proposition (\ref{prop-ss-2}) en niveau $T-S_1$
donne alors:
\begin{multline*}
n^{(s_g,1)}_{(s_g,0),\pi_v}(\overline{\lambda_T} \otimes \varrho^{T,\oo}\otimes \varrho_v
\otimes \chi) + \sum_{S \subsetneq S_1} m_{S}(\overline{\lambda_T} \otimes \varrho^{T,\oo}
\otimes \varrho_v \otimes \chi) +
+ m_{S_1}(\overline{\lambda_T} \otimes \varrho^{T,\oo}
\otimes \varrho_v \otimes \chi) = \\
n^{(s_g,1)}_{(s_g,0),\pi'_v}(\overline{\lambda_T} \otimes \varrho^{T,\oo}\otimes \varrho_v
\otimes \chi) + \sum_{S \subsetneq S_1} m'_{S}(\overline{\lambda_T} \otimes \varrho^{T,\oo}
\otimes \varrho_v \otimes \chi) +
+ m'_{S_1}(\overline{\lambda_T} \otimes \varrho^{T,\oo}
\otimes \varrho_v \otimes \chi) 
\end{multline*}
ce qui, par la bien connue formule du binÙme $(1-1)^r=\sum_{i=0}^r (-1)^i \binom{r}{i}=0$,
fournit 
\begin{multline*}
n^{(s_g,1)}_{(s_g,0),\pi_v}(\overline{\lambda_T} \otimes \varrho^{T,\oo}\otimes \varrho_v
\otimes \chi) -
n^{(s_g,1)}_{(s_g,0),\pi'_v}(\overline{\lambda_T} \otimes \varrho^{T,\oo}\otimes \varrho_v
\otimes \chi) = \\
(-1)^{\sharp S_1} \Bigl ( m'_{S_1}(\overline{\lambda_T} \otimes \varrho^{T,\oo}
\otimes \varrho_v \otimes \chi) -
m_{S_1}(\overline{\lambda_T} \otimes \varrho^{T,\oo} \otimes \varrho_v \otimes \chi) \Bigr )
\end{multline*}
\end{proof}

Comme il est peu probable de construire aussi aisÈment de telles congruences automorphes,
et encore n'avons nous pas poussÈ de telles constructions au maximum,
nous proposons la conjecture suivante Èquivalente ‡ \ref{conj}.

\begin{conj} \label{conj-2}
Pour tout $\pi_v$, pour tout $1 \leq r \leq s_g$ et pour tout $i$, la rÈduction modulo $l$ de la torsion
de $H^i_c(\FC_{\xi}(\pi_v,r))$, ne dÈpend dans 
les $\GF_{\bar \Fm_l,T}(v,d-rg)$ que de $T, i, \xi$ et $r_l(\pi_v)$.
\end{conj}

\rem la torsion construite au \S \ref{para-torsion} vÈrifie bien cette propriÈtÈ

La conjecture \ref{conj-2} est Èquivalente ‡ demander la mÍme propriÈtÈ pour la rÈduction modulo $l$
des parties libres des $H^i_c(\FC_\xi(\pi_v,r))$, Ètant entendu que la partie libre d'un $\bar \Zm_l$-module
est le quotient par son sous-module de torsion.

Pour montrer l'Èquivalence entre les deux conjectures \ref{conj} et \ref{conj-2}, on raisonne
comme dans le paragraphe prÈcÈdent et il s'agit de sÈparer les contributions des diffÈrents 
$s$ dans la formule (\ref{eq-a-distinguer}). Pour ce faire on utilise les parties
libres des $H^i_c(\FC_\xi(r',\pi_v)[d-r'g])$ pour des $r'<r$ et $i>0$. En effet comme
les diagrammes des $n(r',i)$ du \S \ref{para-coho-c}, ne sont pas les mÍmes pour
diffÈrents $s$, on procËde comme suit. Pour $r$ fixÈ, on raisonne par rÈcurrence 
descendante sur $s$ en supposant que pour tout $s>s_0$ 
$$n^{(r,s)}_{(r,0),\pi_v}(\overline{\lambda_T} \otimes \varrho^{T,\oo} \otimes \varrho_v
\otimes \chi) =
n^{(r,s)}_{(r,0),\pi'_v}(\overline{\lambda_T} \otimes \varrho^{T,\oo} \otimes \varrho_v
\otimes \chi).$$
On regarde alors la rÈduction modulo $l$ des parties libres de $H^{s_0-1}_c(\FC_\xi(\pi_v,r-s_0+1)[d-(r-s_0+1)g])$
et $H^{s_0-1}_c(\FC_\xi(\pi'_v,r-s_0+1)[d-(r-s_0+1)g])$. Une fois enlevÈe les contributions des
des $(r',s)$ pour $r'>r$ et celles des $(r,s)$ pour $s \geq s_0$, ‡ l'aide de l'hypothËse de 
rÈcurrence, on obtient
$$n^{(r,s_0)}_{(r,0),\pi_v}(\overline{\lambda_T} \otimes \varrho^{T,\oo} \otimes \varrho_v
\otimes \chi) =
n^{(r,s_0)}_{(r,0),\pi'_v}(\overline{\lambda_T} \otimes \varrho^{T,\oo} \otimes \varrho_v
\otimes \chi).$$

\bibliographystyle{plain}
\bibliography{bib-ok}

\end{document}